\documentclass[11pt]{amsart}
\usepackage{amsmath}
\usepackage{amssymb}
\usepackage{times}
\usepackage{a4wide}
\usepackage{xcolor}
\usepackage{hyperref}

\newtheorem{thm}{Theorem}[section]
\newtheorem{lem}[thm]{Lemma}
\newtheorem{defn}[thm]{Definition}
\newtheorem{prop}[thm]{Proposition}
\newtheorem{cor}[thm]{Corollary}

\newtheorem{rem}[thm]{Remark}

\newtheorem{otherth}{\bf Theorem}

\newtheorem{otherl}{\bf Lemma}

\newcommand{\C}{{\mathbb C}}

\newcommand{\D}{{\mathbb D}}
\newcommand{\R}{{\mathbb R}}

\newcommand{\N}{{\mathbb N}}

\newcommand{\om}{\omega}

\numberwithin{equation}{section}

\title[Volterra type integral operators]{Generalized Volterra type integral operators \\on large Bergman spaces}

\author{H. Arroussi}
\address{H. Arroussi\\Department of Mathematics, University of Reading, England}
\email{arroussihicham@yahoo.fr}

\author{H. Gissy}
\address{H. Gissy\\ Department of Mathematics, University of Reading, England\newline \indent Department of Mathematics, Jazan University, Saudi Arabia}
\email{h.m.h.gissy@pgr.reading.ac.uk, hgissy@jazanu.edu.sa}
\author{J. A. Virtanen}
\address{J. Virtanen\\Department of Mathematics, University of Reading, England\newline \indent Department of Mathematics, University of Helsinki, Finland}
\email{j.a.virtanen@reading.ac.uk, jani.virtanen@helsinki.ac.uk}


\begin{document}
\begin{abstract}
Let $\phi$ be an analytic self-map of the open unit disk $\D$ and $g$ analytic in $\D$. We characterize boundedness and compactness of generalized Volterra type integral operators
$$
	GI_{(\phi,g)}f(z)= \int_{0}^{z}f'(\phi(\xi))\,g(\xi)\, d\xi
$$
and
$$
	GV_{ (\phi, g)}f(z)= \int_{0}^{z} f(\phi(\xi))\,g(\xi)\, d\xi,
$$
acting between large Bergman spaces $A^p_\omega$ and $A^q_\omega$ for $0<p,q\le \infty$. To prove our characterizations, which involve Berezin type integral transforms, we use the Littlewood-Paley formula of Constantin and Pel\'aez and establish corresponding embedding theorems, which are also of independent interest. When $\phi(z) = z$, our results for $GV_{(\phi,g)}$ complement the descriptions of Pau and Pel\'aez.
\end{abstract}

\maketitle

\section{Introduction and main results}
For $0< p < \infty$ and a positive function $\omega\in L^1(\mathbb{D}, dA)$, the weighted Bergman spaces $A^p_{\omega}$ and $A^\infty_\omega$ consist of all analytic functions defined on the unit disk $\D$ for which
\[
	 \|f \|^p_{A^p_{\omega}} = \int_{\D}|f(z)|^p  \,\omega(z)^{p/2}\, dA(z) < \infty
\]
and 
$$
	\|f\|_{A^\infty_\omega} = \sup_{z\in \D} |f(z)|\, \om(z)^{1/2} < \infty,
$$
respectively, where $dA$ is the normalized area measure on $\D$.

In this paper, we study generalized Volterra type integral operators between weighted Bergman spaces for a certain class $\mathcal{W}$ of radial rapidly decreasing weights. The class $\mathcal{W},$  considered previously  in \cite{GP1} and \cite{PP1}, consists of the radial decreasing weights  of the form $\omega(z)=e^{-2\varphi(z)}$,
where $\varphi\in C^2(\D)$ is a radial function such that
$\left(\Delta\varphi(z)\right)^{-1/2}\asymp \tau(z)$ for some  radial positive function $\tau(z)$ that decreases to $0$ as $|z|\rightarrow 1^{-}$ and satisfies  $\lim_{r\to 1^-}\tau'(r)=0.$ Here $\Delta$ denotes the standard Laplace operator.  Furthermore, we assume that  there either exists  a constant $C>0$  such that
$\tau(r)(1-r)^{-C}$ increases for $r$ close to $1$ or
$$
	 \lim_{r\to 1^-}\tau'(r)\log\frac{1}{\tau(r)}=0.
$$
See Section 7 of \cite{PP1} for examples of weights in $\mathcal{W}$, such as the  following exponential type weight
$$
	\omega_{\gamma, \alpha}(z) 
	= (1-|z|)^\gamma\exp\left(\frac{-b}{(1-|z|)^\alpha}\right),
	\quad  \gamma\ge 0, \alpha>0, b>0.
$$
For the weights $\omega$ in $\mathcal{W}$, the point evaluations $L_ z: f\longmapsto f(z)$ are bounded linear functionals on $A^2_{\omega}$ for each $z\in \D$, and so $A^2_{\omega}$ is a reproducing kernel Hilbert space; that is, for each $z\in \D$, there are functions $K_ z\in A^2_{\omega}$ with $\|L_ z\|=\|K_ z\|_{A^2_{\omega}}$ such that $L_ z f=f(z)=\langle f, K_ z \rangle _{\omega}$, where
$$
	\langle f,g \rangle_{\omega}=\int_{\D} f(z)\,\overline{g(z)} \,\omega(z)\,dA(z).
$$
The function $K_z$ is called the reproducing kernel for the Bergman space $A^2_{\omega}$ and has the property that $K_ z(\xi)=\overline{K_{\xi}(z)}$. The Bergman spaces with exponential type weights have attracted considerable attention in recent years because of novel techniques different from those used for standard Bergman spaces; see, e.g., \cite{BDK, xiaof} and the references therein. Various estimates for the reproducing kernel play an important role in our work and we discuss them further in Section~\ref{sec2:1}.

Let $\phi$ and $g$ be  analytic self-maps of $\D$.   The generalized Volterra type integral operators $GI_{(\phi, g)}$ and $GV_{(\phi, g)}$ induced by the pair of symbols  $(\phi, g)$ are defined by
\begin{equation}\label{e:GI and GV}
	GI_{(\phi,g)}f(z) = \int_{0}^{z}f'(\phi(\xi))\,g(\xi) d\xi\quad {\rm and}\quad
	GV_{ (\phi, g)}f(z) = \int_{0}^{z} f(\phi(\xi))\,g(\xi) d\xi, 
\end{equation}
where $ f\in H(\D)$ and  $z\in \D$. When $g=\phi',$ the operator $GI_{(\phi,\phi')}$ is the composition operator $C_{\phi}$ up to a certain constant---these operators acting between different large Bergman spaces were recently studied in~\cite{Hi-w}. As another special case, when $\phi(\xi)=\xi$, we obtain the Volterra integral operator
\begin{equation}\label{e:V}
	V_g f(z):=GV_{ (\phi, g')}f(z)= \int_{0}^{z} f(\xi)\,g'(\xi) d\xi,
\end{equation}
and its companion integral operator
\begin{equation}\label{e:J}
	J_gf(z):=GI_{(\phi,g)}f(z)= \int_{0}^{z}f'(\xi)\,g(\xi) d\xi.
\end{equation}
Previously Pau and Pel\'{a}ez characterized boundedness and compactness of $V_g : A^p_{\omega} \to A^q_{\omega}$ in~\cite{PP1} when $0<p,q<\infty$. Via~\eqref{e:V} and~\eqref{e:J}, our characterizations extend the previous results to the full range $0<p,q\le\infty$ and to all weights in $\mathcal{W}$, and also deal with the companion operator $J_g$ for the first time. The generalized Volterra type integral operators $GI_{(\phi, g)}$ and $GV_{(\phi, g)}$ were previously studied by Mengestie~\cite{M2016-1, M2016, MU2019} in standard Fock spaces and by Li~\cite{Li} in standard Bergman spaces and Bloch type spaces.

\subsection{Main results}
In this paper we study boundedness and and compactness of the generalized Volterra type integral operators $GI_{(\phi, g)}$ and $GV_{(\phi, v)}$. Our results on Schatten class properties, compact differences, and the essential norm of these operators will be published elsewhere.

For $0<p,q < \infty$, our characterizations for boundedness and compactness are given in terms of the integral transform
$$
	GB^\phi_{n,p,q}(g)(z) 
	= \int_{\D} |k_{p,z}(\phi(\xi))|^q\,\frac{(1+\varphi'(\phi(\xi))^{nq}}{(1+\varphi'(\xi))^{q}}\,|g(\xi)|^q\,\omega(\xi)^{q/2}\, dA(\xi),\quad z\in \D,
$$
where $n = 0,1$ and $k_{p,z}$ is the normalized reproducing kernel of $A^p_{\om}$.  

\begin{thm} \label{thm1}
Let $\omega \in \mathcal{W}$, $\phi : \D\to \D$ be analytic, and $g\in H(\D)$.

{\rm (A)} For  $0<p\le q < \infty,$ the operator 
$GI_{(\phi,g)}: A^p_\omega\rightarrow A^q_\omega$ is bounded if and only if 
$$
	GB^\phi_{1,p,q}(g)\in L^{\infty}(\D,dA),
$$
and compact if and only if $\lim_{|z|\to1^-}GB^\phi_{1,p,q}(g)(z)=0$.

{\rm (B)} For $0<p<\infty$, $GI_{(\phi,g)}: A^p_\omega\rightarrow A^\infty_\omega$ is bounded if and only if
\begin{equation}\label{eq-G2}
	MI_{g,\phi,\omega}(z):=|g(z)|\, \,\frac{(1+\varphi'(\phi(z))}{(1+\varphi'(z))}\frac{\omega(z)^{1/2}}{\omega(\phi(z))^{1/2}}\Delta\varphi(\phi(z))^{1/p}\in L^{\infty}(\D,dA),
\end{equation}
and compact if and only if $\lim_{|\phi(z)|\to 1^-}MI_{g,\phi,\omega}(z)=0$.

{\rm (C)} The operator $GI_{(\phi,g)}: A^\infty_\omega\rightarrow A^\infty_\omega$ is bounded if and only if
\begin{equation}\label{eq-G3-NI}
	NI_{g,\phi,\omega}(z):=|g(z)|\, \,\frac{(1+\varphi'(\phi(z))}{(1+\varphi'(z))}\frac{\omega(z)^{1/2}}{\omega(\phi(z))^{1/2}}\in L^{\infty}(\D,dA),
\end{equation}
and compact if and only if $\lim_{|\phi(z)|\to 1^-} NI_{g,\phi,\omega}(z)=0$.

{\rm (D)} For  $0<q<p \le \infty$, both boundedness and compactness of $GI_{(\phi,g)}: A^p_\omega\rightarrow A^q_\omega$ are equivalent to the condition 
$$
	GB^\phi_{1,p,q}(g)\in L^{s}(\D,d\lambda),
$$
where $\lambda(z)=dA(z)/\tau(z)^2$ and $s=p/(p-q)$ if $p<\infty$ and $s=1$ if $p=\infty$.
\end{thm}

\begin{thm}\label{thmV1}
Let $\om \in \mathcal{W},$  $\phi:\D\to\D$ be analytic, and $g\in H(\D)$.

{\rm (A)} For $ 0<p\le q< \infty$, $GV_{ (\phi, g)}: A^p_\omega\rightarrow A^q_\omega$ is bounded if and only if $GB^\phi_{0,p,q}(g)\in L^{\infty}(\D,dA)$, and compact if and only if $\lim_{|z|\to1^-}GB^\phi_{0,p,q}(g)=0$.

{\rm (B)} For $ 0<p<\infty,$ the  operator $GV_{ (\phi, g)}: A^p_\omega\rightarrow A^\infty_\omega$ is bounded if and only if 
\begin{equation}\label{eq-V2}
	MV_{g,\phi,\omega}(z):=\frac{|g(z)|}{(1+\varphi'(z))}\frac{\omega(z)^{1/2}}{\omega(\phi(z))^{1/2}}\Delta\varphi(\phi(z))^{1/p}\in L^{\infty}(\D,dA),
\end{equation}
and compact if and only if $\lim_{|\phi(z)|\to 1} MV_{g,\phi,\omega}(z) =0$.

{\rm (C)} The operator $GV_{ (\phi, g)}: A^\infty_\omega\rightarrow A^\infty_\omega$ is bounded if and only if 
\begin{equation}\label{eq-G3-NV}
	NV_{g,\phi,\omega}(z):=\frac{|g(z)|}{(1+\varphi'(z))}\frac{\omega(z)^{1/2}}{\omega(\phi(z))^{1/2}}\in L^{\infty}(\D,dA),
\end{equation}
and compact if and only if $\lim_{|\phi(z)|\to 1} NV_{g,\phi,\omega}(z) = 0$.

{\rm (D)} For $ 0<q< p \le \infty$, both boundedness and compactness of $GV_{ (\phi, g)}: A^p_\omega\rightarrow A^q_\omega$ are equivalent to the condition
$$
	GB^\phi_{0,p,q}(g)\in L^r(\D,d\lambda),
$$
where $r = p/(p-q)$ when $p<\infty$ and $r=1$ when $p=\infty$.
\end{thm}

We also prove the following simpler necessary conditions for boundedness and compactness.

\begin{prop}\label{prop1}
Let $\omega \in \mathcal{W}$, $\phi : \D\to \D$ be analytic, and $g\in H(\D)$. 

{\rm (A)} If $0<p,q<\infty$ and $GI_{(\phi,g)}: A^p_\omega\rightarrow A^q_\omega$ is bounded, then 
\begin{equation}\label{eq-nc1-GI}
	z\mapsto |g(z)|\frac{\tau(z)^{2/q}}{\tau(\phi(z))^{2/p}}\frac{(1+\varphi'(\phi(z))}{(1+\varphi'(z))}\frac{\omega(z)^{1/2}}{\omega(\phi(z))^{1/2}}
	\in L^{\infty}(\D,dA);
\end{equation}
and if $GI_{(\phi,q)}$ is compact, then the function in \eqref{eq-nc1-GI} vanishes as $|z|\to 1$.

{\rm (B)}  If $0<p\le q< \infty$ and $GV_{(\phi,g)} : A^p_\omega\rightarrow A^q_\omega$ is bounded, then 
\begin{equation}\label{eq-nc1-GV}
	z\mapsto \frac{\tau(z)^{2	/q}}{\tau(\phi(z))^{2/p}}\frac{|g(z)|}{(1+\varphi'(z))}
	\frac{\omega(z)^{1/2}}{\omega(\phi(z))^{1/2}}\in L^{\infty}(\D,dA);
\end{equation}
and if $VG_{(\phi, g)}$ is compact, then the function in~\eqref{eq-nc1-GV} vanishes as $|\phi(z)|\to 1$.
\end{prop}

As a consequence of the two main theorems, when $\phi(z) = z$, we obtain characterizations for boundedness and compactness of $V_g$ and its companion operator $J_g$. The results for $J_g$ are new while the descriptions for $V_g$ had been partially obtained before as explained in the following remark.

\begin{rem} Notice that {\rm (C)} and {\rm (D)} of Corollary~\ref{cor1} are analogous to the descriptions given in Theorem~3 of Constantin and Pel\'aez~\cite{CoPe} when $V_g$ is acting between weighted Fock spaces, but the two cases require different methods due to fundamental differences between the two types of spaces. Further, Corollary \ref{cor2} implies Theorem~2 of Pau and Pel\'aez~\cite{PP1}, that is, we show that the their conditions are equivalent of those in Theorem~\ref{thmV1} when $\phi(z) = z$.
\end{rem}

\begin{cor}\label{cor1}
Let $\omega \in \mathcal{W}$ and $g\in H(\D)$.

{\rm (A)} For $0<p<q\le \infty$, $J_g : A^p_\omega\to A^q_\omega$ is bounded if and only if $g=0$. 

{\rm (B)} For $p>q$, $J_g : A^p_\omega\to A^q_\omega$ is compact if and only if $g\in L^s(\D, dA)$, where $s=pq/(p-q)$.

{\rm (C)} For $0<p\le q \le \infty$, $V_g : A^p_\omega \to A^q_\omega$ is bounded if and only if
\begin{equation}\label{eq-V7}
	z\mapsto \frac{|g'(z)|}{(1+\varphi'(z))}\Delta\varphi(z)^{\frac{1}{p}-\frac{1}{q}}\in L^{\infty}(\D,dA),
\end{equation}
and $V_g : A^p_\omega \to A^q_\omega$ is compact if the function in \eqref{eq-V7} vanishes as $|z|\to 1$.

{\rm (D)} For $0<q<p<\infty$, $V_g : A^p_\omega \to A^q_\omega$ is bounded if and only if
\begin{equation}\label{eq-V9}
	\frac{|g'(z)|}{(1+\varphi'(z))}\in L^{\frac{pq}{p-q}}(\D,dA).
\end{equation}
\end{cor}

In the next corollary, we consider the weighted Bergman space $A^p(\omega) = A^p_{\omega^{2/p}}$, that is, 
$$
	A^p(\omega) = \left\{ f\in H(\D) : \|f\|_{A^p(w)}^p = \int_{\D} |f(z)|^p\, \omega(z)\, dA(z)<\infty\right\},
$$
where the weight $\omega\in\mathcal{W}$ satisfies the condition
$$
\Delta\phi(z)\asymp ((1-|z|)^{t}\psi_\om(z))^{-1}, \,\, z\in \D, \,\, \textrm{for some}\,\,\, t\geq 1.
$$
In particular, we obtain the conditions of Theorem 2 in~\cite{PP1} for boundedness and compactness of the operator $V_g : A^p(\omega) \to A^q(\omega)$.

\begin{cor}\label{cor2} Let  $ 0<p,q< \infty$, $\om \in \mathcal{W}$, and $g\in H(\D)$.
\begin{enumerate} 

\item[(I)] For  $p=q,$  we have the following statements
\begin{enumerate} 
 \item[(a)]  $GB^{id}_{0,p,q}(g') \in L^{\infty}(\D,dA)$ \quad if and only if $$\psi_\om(z)|g'(z)|\in L^{\infty}(\D,dA).$$
\item[(b)]  $\lim_{|z|\to1} GB^{id}_{0,p,q}(g')=0$ if and only if 
$$\lim_{|z|\to1}\psi_\om(z)|g'(z)|=0.$$
\end{enumerate}
\item[(II)] Let $\om \in \mathcal{W}$ with 
\begin{equation}\label{jcn}
\Delta\phi(z)\asymp ((1-|z|)^{t}\psi_\om(z))^{-1}, \,\, z\in \D, \,\, \textrm{for some}\,\,\, t\geq 1.
\end{equation}
 For  $p<q,$  the following statements are equivalent:
\begin{enumerate} 
\item[(c)]  $GB^{id}_{0,p,q}(g') \in L^{\infty}(\D,dA).$
\item[(d)]  The function $g$ is constant.
\end{enumerate}
\item[(III)] For  $q<p$,
$$
	GB^{id}_{0,p,q}(g')\in L^{p/(p-q)}(\D,d\lambda)
	\implies g\in A^{pq/(p-q)}(\om).
$$
\end{enumerate}
\end{cor}

\subsection{Outline}
In Section~\ref{preliminaries} we provide the basic definitions and results that are needed to deal with the weights $\omega$ in $\mathcal{W}$, and consider useful estimates for the reproducing kernel of $A^p_\omega$.  In Section~\ref{carleson}, we recall known geometric characterizations of Carleson measures, and in Section~\ref{embedding} we establish embedding theorems of $S^p_\omega$ into $L^q(\D,d\mu),$ for  $0<p,q\le \infty$ and $\omega \in \mathcal{W},$
where
\begin{equation}\label{e:S^p}
	S^p_\omega:=\Bigg\{f\in H(\D) : \|f\|_{S^p_\omega} = \int_{\D}|f(z)|^p\,\frac{\omega(z)^{p/2}}{(1+\varphi'(z))^p} \,{dA(z)<\infty}\Bigg\}
\end{equation}
and
\begin{equation}\label{e:S^infty}
	S^\infty_\omega:=\Bigg\{f\in H(\D) : \|f\|_{S^\infty_\omega}=\sup_{z\in\D}|f(z)|\,\frac{\omega(z)^{1/2}}{1+\varphi'(z)} \,{<\infty}\Bigg\}.
\end{equation}

In Section \ref{proofs1}, we prove Theorems \ref{thm1} and \ref{thmV1} using the embedding theorems, the strong decay of the weights $e^{-2\varphi}$ and the following Littlewood-Paley type formulas (see (9.3) of \cite{CoPe} and \cite{MU2019}):
\begin{align}\label{littleW}
	\|f \|^p_{A^p_{\omega}} &\asymp |f(0)|+  \int_{\D}|f'(z)|^p  \,\frac{\omega(z)^{p/2}}{(1+\varphi'(z))^p}\, dA(z),\\
	\label{littleW1}
	\|f \|_{A^\infty_{\omega}} &\asymp |f(0)|+  \sup_{z\in\D}|f'(z)|  \,\frac{\omega(z)^{1/2}}{(1+\varphi'(z))}.
\end{align}
Finally, Proposition \ref{prop1} and Corollaries \ref{cor1} and \ref{cor2} are proved in Section \ref{proofs2}.

\medskip

Throughout the paper, we use the notation $a \lesssim b$ to indicate that
there is a constant $C > 0$ with $a \le C b$. By $a \asymp b$ we mean that $a \lesssim b$ and $b \lesssim a$. For simplicity, we write $L^p_\om$ and $A^p_{\om}$  for $L^p(\D,\om^{p/2}\,dA)$ and $A^p(\D,\om^{p/2}\, dA)$, respectively.

\section{Preliminaries and basic properties}\label{preliminaries}

A positive function $\tau$ on $\D$ is said to be of class $\mathcal{L}$ if there are two constants $c_1$ and $c_2$ such that
\begin{equation}\label{e:A}
	\tau(z)\le c_ 1\,(1-|z|)\quad \text{\rm for all}\ z\in\D
\end{equation}
and
\begin{equation}\label{e:B}
	|\tau(z)-\tau(\zeta)|\le c_ 2\,|z-\zeta| \quad \text{\rm for all}\ z,\zeta\in \D.
\end{equation}
For such $c_1$ and $c_2$, we set 
$$
	m_\tau : =  \tfrac14{\min(1, c_1^{-1}, c_2^{-1})}.
$$
Given $a\in \D$ and $\delta>0$, we denote by $D_\delta(a)$ the Euclidean disc centered at $a$ with radius $\delta \tau(a)$. It follows from~\eqref{e:A} and~\eqref{e:B} (see \cite[Lemma 2.1]{PP1}) that if $\tau \in\mathcal{L}$  and $ z \in  D_\delta(a),$ then
\begin{equation}\label{eqn:asymptau}
	\frac{1}{2}\,\tau(a)\leq \tau(z) \leq  2 \,\tau(a),
\end{equation}
whenever $\delta \in (0, m_\tau).$ These inequalities will be used frequently in what follows.

\begin{defn}\label{definit}
We say that a weight $\omega$ is of \emph{class  $\mathcal{L}^*$} if it is of the form $\omega =e^{-2\varphi}$, where $\varphi \in C^2(\D)$ with $\Delta \varphi>0$, and $\big (\Delta \varphi (z) \big )^{-1/2}\asymp \tau(z)$, with $\tau$ being a function in the class $\mathcal{L}$. Here $\Delta$ denotes the classical Laplace operator.
\end{defn}

It is straightforward to see that $\mathcal{W}\subset \mathcal{L}^*$. The following result (see~\cite[Lemma 2.2]{PP1}) implies that the point evaluation functional at each $z\in\D$ is bounded on $A^2_{\omega}$.

\begin{otherl}\label{lem:subHarmP}
Let $\omega \in \mathcal{L}^*$, $0<p<\infty$, and  $z\in \D$. If $\beta\in \mathbb{R},$ there exists $ M \geq 1$  such that
\[ |f(z)|^p \omega(z)^{\beta} \leq  \frac{M}{\delta^2\tau(z)^2}\int_{D_\delta(z)}{|f(\xi)|^p \omega(\xi)^{\beta} \   dA(\xi)}\]
for all $ f \in H(\mathbb{D})$  and all sufficiently small $\delta > 0 .$ 
\end{otherl}

Using the preceding lemma and the fact that there exists $r_0\in [0,1)$ such that for all $a\in \D$ with $1>|a|>r_0,$ and any $\delta>0$ small enough we have 
$$
	\varphi'(a)\asymp \varphi'(z), \quad z\in D_\delta(a)
$$
(see statement (d) in \cite[Lemma 32]{CoPe}), one has 
\begin{equation}\label{Eq-gamma}
|f(z)|^p \frac{\omega(z)^{\beta}}{(1+\varphi'(z))^{\gamma}} \lesssim  \frac{1}{\delta^2\tau(z)^2}\int_{D_\delta(z)}{|f(\xi)|^p \frac{\omega(\xi)^{\beta}}{(1+\varphi'(\xi))^{\gamma}}  dA(\xi)}, 
\end{equation}
for $\beta, \gamma \in \mathbb{R}$.

The next lemma provides upper estimates for the derivatives of functions
in $A^p_\om.$  Its proof is similar to the case of doubling measures $\Delta\varphi$ in Lemma 19 of \cite{MMO}, and it can be found in the following form in~\cite{xiaof, O}.

\begin{otherl}\label{DERI}
Let $ \omega\in\mathcal{L}^*$ and $0 < p < \infty.$  For any $\delta_0 > 0$ sufficiently small  there exists a constant $C(\delta_0)> 0$ such that $$|f'(z)|^{p}\omega(z)^{p/2} \leq \frac{C(\delta_0)}{\tau(z)^{2+p}}\bigg(\int_{D(\delta_0\tau(z)/2)}{|f(\xi)|^p \,\omega(\xi)^{p/2} dA(\xi)}\bigg)^{1/p}, $$
for all $z\in \D$ and $f \in H(\mathbb{D})$.
\end{otherl}

\par The following lemma on coverings is due to
Oleinik~\cite{O}.

\begin{otherl}\label{lem:Lcoverings}
Let $\tau$ be a positive function on $\D$ of class
$\mathcal{L}$, and let $\delta\in (0,m_{\tau})$. Then there exists
a sequence of points $\{z_ n\}\subset \D$ such that the following
conditions are satisfied:
\begin{enumerate}
\item[$(i)$] \,$z_ n\notin D_\delta(z_ k)$, \,$n\neq k$.

\item[$(ii)$] \, $\bigcup_ n D_\delta(z_ n)=\D$.

\item[$(iii)$] \, $\tilde{D}_\delta(z_ n)\subset
D_{3\delta}(z_ n)$, where $\tilde{D}_\delta (z_
n)=\bigcup_{z\in D_\delta(z_ n)}D_\delta(z)$,
$n\in \N$.

\item[$(iv)$] \,$\big \{D_{3\delta}(z_ n)\big \}$ is a
covering of $\D$ of finite multiplicity $N$.
\end{enumerate}
\end{otherl}
The multiplicity $N$ in the previous lemma  is independent of $\delta$, and it is easy to see that one can take, for example, $N=256$. Any sequence satisfying the conditions in Lemma \ref{lem:Lcoverings} will be called a $(\delta,\tau)$-\emph{lattice}. Note that $|z_n|\to 1^{-}$ as $n\to \infty.$ In what follows, the sequence $\{z_n\}$ will always refer to the sequence chosen in Lemma \ref{lem:Lcoverings}.

\subsection{Reproducing kernel estimates }\label{sec2:1}
The following norm estimates for the reproducing kernel $K_z$ valid for all $z\in\D$ can be found in \cite{BDK,LR1,PP1} when $p=2$ and in \cite{xiaof} when $p>0$, while for the estimate for the points close to the diagonal, see \cite[Lemma 3.6]{LR2}.

\begin{otherth}\label{RK-PE}
Let $K_z $ be the reproducing kernel of  $A^2_{\omega}.$ Then
\begin{enumerate}
\item[(a)] For  $\om\in \mathcal{W}$ and $0< p< \infty$, one has
\begin{equation}\label{Eq-NE}
\|K_z\|_{A^p_{\omega}}   \asymp  \omega(z)^{-1/2}\, \tau(z)^{2(1-p)/p},\qquad z\in \D.
\end{equation}
\begin{equation}\label{Eq-INF}
\|K_z\|_{A^\infty_{\omega}}   \asymp  \omega(z)^{-1/2}\, \tau(z)^{-2},\qquad z\in \D.
\end{equation}
\item[(b)] For all sufficiently small $\delta \in (0,m_{\tau})$ and $\om \in \mathcal{W},$  one has
\begin{equation}\label{RK-Diag}
|K_ z(\zeta)| \asymp \|K_ z\|_{A^2_{\omega}}\cdot \|K_{\zeta}\|_{A^2_{\omega}} ,\qquad \zeta \in D_\delta(z).
\end{equation}
\end{enumerate}
\end{otherth}

The next lemma generalizes the statement $(a)$ of the above theorem. For the proof, see~\cite{Hi-w}.

\begin{otherl}\label{nEstim}
Let $K_z $ be the reproducing kernel of  $A^2_{\omega}$ where $\omega$ is a weight in  the class $\mathcal{W}$. For each $z\in \D,$  $0< p < \infty$ and $\beta \in \R,$ one has
\begin{equation}\label{Eq-NE1}
  \int_{\D} |K_{z}(\xi)|^p\,\omega(\xi)^{p/2}\,\tau(\xi)^{\beta}\,dA(\xi) \asymp \om(z)^{-p/2}\, \tau(z)^{2(1-p)+\beta}.
\end{equation}
\end{otherl}

The following result gives estimates for the normalized reproducing kernel $k_{p,z}$ in $A^p_{\om}$ defined by
$$
	k_{p,z}= K_z/\|K_z\|_{A^p_{\om}}
$$
for $z\in \D$.

\begin{lem}\label{lem:RK-PE1}
Let $\omega \in \mathcal{W}$. Then
\begin{enumerate}
\item[(a)] For each $z\in \D,$ $0< p\le \infty,$ and  $0< q < \infty,$ 
\begin{equation}\label{eqn:Eq-NE1}
|k_{p,z}(\zeta)|^q   \asymp \tau(z)^{2(1-\frac{q}{p})}|k_{q,z}(\zeta)|^q,\qquad \zeta\in \D.
\end{equation}
\item[(b)] For $q=\infty,$
$$
	|k_{p,z}(\zeta)|  \asymp \tau(z)^{-2/p}|k_{q,z}(\zeta)|,\qquad \zeta\in \D.
$$
\item[(c)] For all $\delta \in (0,m_{\tau})$ sufficiently small,
\begin{equation}\label{eqn:RK-Diag1}
|k_ {p,z}(\zeta)|^p \, \omega(\zeta)^{p/2} \asymp \tau(z)^{-2} ,\qquad \zeta \in D_\delta(z).
\end{equation}
\end{enumerate}
\end{lem}

\begin{proof} 
The proof is immediate from Theorem~\ref{RK-PE}.
\end{proof}

\subsection{Test functions and some estimates }\label{sec2:2}
The following result on test functions was obtained in  \cite{PP1} and Lemma 3.3 in \cite{BDK}. Without loss of generality, we modified the original version by taking $\omega(z)^{p/2}$ instead of $\omega(z)$ when $0<p<\infty.$ 

\begin{otherl}\label{Borichevlemma}
Let $n\in\mathbb{N}\setminus\{0\}$ and $\omega\in \mathcal{W}$. There is a number $\rho_0 \in (0,1)$ such that for each $a\in \D$ with $|a|>\rho_0$
there is a function $F_{a,n}$ analytic in $\D$ with
\begin{equation}\label{BL1}
|F_{a,n}(z)|\,\omega(z)^{1/2}\asymp 1 \quad \textrm{ if }\quad
|z-a|<\tau(a),
\end{equation}
and
\begin{equation}\label{BL2}
|F_{a,n}(z)|\,\omega(z)^{1/2}\lesssim \min \left
(1,\frac{\min\big(\tau(a),\tau(z)\big)}{|z-a|} \right )^{3n}, \quad
z\in \D.
\end{equation}
Moreover,  
\begin{enumerate}
\item[(a)] For $0<p<\infty,$ the function $F_{a,n}$  belongs to $A^p(\omega)$  with $$\|F_{
a,n}\|_{A^p_\om}\asymp \tau(a)^{2/p}.$$
\item[(b)]  For $p=\infty,$ the function $F_{a,n}$  belongs to $A^\infty_{\omega}$  with $$\|F_{
a,n}\|_{A^\infty_{\omega}}\asymp 1.$$ 
\end{enumerate}
\end{otherl}

As a consequence, we have the following pointwise estimates for the derivative of the test functions $F_{a,n}$. Its proof is a simple application of \eqref{BL1}.

\begin{lem}\label{BL3}
Let $n\in\mathbb{N}\setminus\{0\}$ and $\omega\in \mathcal{W}$.
For any $\delta>0$ small enough,
\begin{equation}
|F'_{a,n}(z)|\,\omega(z)^{1/2}\asymp 1+\varphi'(z), \quad z \in D_{\delta}(a).
\end{equation}
\end{lem}

The next Proposition is a partial result about the atomic decomposition on $A^p_\om$ and its proof follows easily from Lemma \ref{Borichevlemma}.

\begin{prop}\label{At1-pp}
Let $n\geq 2$  and $\omega \in \mathcal{W}.$ Let  $\{z_k\}_{k\in\mathbb{N}}\subset\D$ be the sequence  defined in Lemma \ref{lem:Lcoverings}.
\begin{enumerate}
\item[(a)] For  $0 < p < \infty,$ the function given by
$$ F(z): = \sum_{\substack{k = 0}}^{\infty}\lambda_k \,\,\frac{F_{z_k,n}(z)}{\tau(z_k)^{2/p}}$$ belongs to $A^p_\om$ for every sequence $\lambda=\lbrace\lambda_k\rbrace\in \ell^{p}$ . Moreover,$$\|F\|_{A^p_\om}\lesssim \|\lambda \|_{\ell^p}.$$
\item[(b)] For $p=\infty,$ the function given by
 $$ F(z): = \sum_{\substack{k = 0}}^{\infty}\lambda_k \,F_{z_k,n}(z)$$ belongs to $A^\infty_\om$ for every sequence $\lambda=\lbrace\lambda_k\rbrace\in \ell^{\infty}$ . Moreover,$$\|F\|_{A^\infty_{\omega}}\lesssim \|\lambda \|_{\ell^{\infty}}.$$
\end{enumerate}
\end{prop}

\begin{proof}
The proof of $(a)$ can be found in \cite[Proposition 2]{PP1}. To prove $(b)$, estimate the norm of $F$ as follows
\begin{align*}
	\|F\|_{A^\infty_{\omega}} &=\sup_{z\in \D}\,|F(z)|\omega(z)^{1/2}\\
	&\lesssim \|\lambda\|_{\ell^\infty}\sum_{\substack{k = 0}}^{\infty}  |F_{z_k,n}(z)| \omega(z)^{1/2}\\
	&=\|\lambda\|_{\ell^\infty}\left(\sum_{\substack{z_k \in D_{\delta}(z)}}  |F_{z_k,n}(z)| \omega(z)^{1/2}+ \sum_{\substack{z_k \notin D_{\delta}(z)}}  |F_{z_k,n}(z)| \omega(z)^{1/2}\right)
\end{align*}
Now, using \eqref{BL1} and (iv) of Lemma \ref{lem:Lcoverings}, we have 
\begin{equation}\label{eq-sum1}
\begin{split}
\sum_{\substack{z_k \in D_{\delta}(z)}}  |F_{z_k,n}(z)| \omega(z)^{1/2}
&\lesssim 1.
\end{split}
\end{equation}
It remains to show that 
$$
	\sum_{\substack{z_k \notin D_{\delta}(z)}}
	|F_{z_k,n}(z)| \omega(z)^{1/2} \lesssim 1.
$$
Indeed, by  H\"older's inequality, we have 
\begin{equation}\label{eq-sum2}
	\sum_{\substack{z_k \notin D_{\delta}(z)}}  |F_{z_k,n}(z)| \omega(z)^{1/2}\le I(z)\cdot II(z),
\end{equation}
where
$$
	I(z)=\sum_{\substack{z_k \notin D_{\delta}(z)}} \min\big(\tau(z_k),\tau(z)\big)^{2} |F_{z_k,n}(z)| \omega(z)^{1/2},
$$
and 
$$
	II(z)=\sum_{\substack{z_k \notin D_{\delta}(z)}}\frac{|F_{z_k,n}(z)| \omega(z)^{1/2}}{\min\big(\tau(z_k),\tau(z)\big)^{2}}.
$$
First we look for the upper bound of $I(z).$  To do this, we need to consider  the covering of $\displaystyle\lbrace \xi\in\mathbb{D} : |z-\xi| > \delta \tau(z)\rbrace$ given by 
$$
	R_j(z) =\displaystyle\lbrace\xi\in\mathbb{D} : 2^{j}\delta\tau(z) < |z-\xi| \leq 2^{j+1}\delta\tau(z)\rbrace,\qquad j=0,1,2,\dots
$$
and observe that, using $(A)$ of properties of $\tau$, it is easy to see that, for  $j=0,1,2,\dots,$
$$
	D_{\delta}(z_k)\subset D_{r}(z), \,  \,  \textrm{  if } \,  \,  z_k \in D_{t}(z)\,  \,  \textrm{with} \,  \,      r= 5\delta2^{j} \, \,    \textrm{and} \,\, t= \delta2^{j+1} .
$$
This fact together with the finite multiplicity of the covering (see Lemma \ref{lem:Lcoverings}) gives
\begin{equation*}\label{eq-sum3}
\begin{split}
\sum_{\substack{z_k \in R_j(z)}}\tau(z_k)^{2}\lesssim A(D_{r}(z))\lesssim2^{2j}\tau(z)^{2}.
\end{split}
\end{equation*}
Therefore, by \eqref{BL2}, we have 
\begin{equation}\label{eq-sum4}
\begin{split}
I(z)&\le\sum_{\substack{z_k \notin D_{\delta}(z)}} \tau(z_k)^{2} |F_{z_k,n}(z)| \omega(z)^{1/2}\\
&\lesssim\sum_{\substack{z_k \notin D_{\delta}(z)}} \tau(z_k)^{2}\min \left
(1,\frac{\min\big(\tau(z_k),\tau(z)\big)}{|z-z_k|} \right )^{3n}\\
&\le \tau(z)^{3n}\sum_{j=0}^{\infty}\sum_{\substack{z_k \in R_j(z)}}\frac{ \tau(z_k)^{2}}{|z-z_k|^{3n}}\\
&\lesssim  \sum_{j=0}^{\infty}2^{-3nj}\sum_{\substack{z_k \in R_j(z)}} \tau(z_k)^{2}\\
&\lesssim \tau(z)^{2} \sum_{j=0}^{\infty}2^{(2-3n)j}\lesssim\tau(z)^{2}.
\end{split}
\end{equation}
To obtain an upper estimate for $(II)$, notice that since $n\geq 2$, \eqref{BL2} implies that 
\begin{equation*}\label{eq-sum5}
\begin{split}
II(z)
&\lesssim\sum_{\substack{z_k \notin D_{\delta}(z)}} 
\frac{\min\big(\tau(z_k),\tau(z)\big)^{3n-2}}{|z-z_k|^{3n}} \\
&\le \tau(z)^{3n-4}\sum_{j=0}^{\infty}\sum_{\substack{z_k \in R_j(z)}}\frac{ \tau(z_k)^{2}}{|z-z_k|^{3n}}\\
&\lesssim \tau(z)^{-4} \sum_{j=0}^{\infty}2^{-3nj}\sum_{\substack{z_k \in R_j(z)}} \tau(z_k)^{2}\\
&\lesssim \tau(z)^{-2} \sum_{j=0}^{\infty}2^{(2-3n)j}\lesssim\tau(z)^{-2}.
\end{split}
\end{equation*}
Combining this and \eqref{eq-sum4} with \eqref{eq-sum2} completes the proof.\end{proof}

\section{Geometric Characterizations of  Carleson measures}\label{carleson}
Let $\mu$ be  a positive measure on $\D.$ Denote by $\widehat{\mu_{\delta}}$  the averaging function defined as
$$\widehat{\mu_{\delta}}(z)= \mu(D_{\delta}(z))\cdot\tau(z)^{-2}, \quad z\in \D,$$ 
and define the general Berezin transform of $\mu$ by 
$$
	G_{t}(\mu)(z) = \int_{\D} |k_{t,z}(\zeta)|^t\, \om(\zeta)^{t/2}\, d\mu(\zeta),
$$
for every $t>0$ and $z\in \D.$

In this section we recall recent characterizations of $q$-Carleson measures for $A^p_{\om}$ for any $0< p, q \le \infty$ in terms of the averaging function $\widehat{\mu_{\delta}}$ and the general Berezin transform $G_t(\mu)$.  For the proofs of all theorems in this section, see Section 3 of \cite{Hi-w}.

\subsection{Carleson measures }\label{sec3:1}
We begin with the definition of $q$-Carleson measures.

\begin{defn}
Let $\mu$ be a positive measure on $\D$ and fix  $0< p,q< \infty.$ We say  that  $\mu$ is  a $q$-Carleson measure for $A^p_{\om}$  if the inclusion $I_{\mu}:A^p_{\om}\longrightarrow L^q_\om$ is bounded.
\end{defn}

The following theorem characterizes the $q$-Carleson measures when $0 < p\le q<\infty$.

\begin{otherth}\label{thm:CMPQ}
Let $\mu$  be a finite positive Borel measure on $\D$. Assume $0<p \le q<\infty,$ $s=p/q,$ $0 < t< \infty.$ The following conditions are  equivalent:
\begin{enumerate}
\item[(a)]  The measure $\mu$ is a $q$-Carleson measure for $A^p_\om$.

\item[(b)]  The function  $$ \tau(z)^{2(1-1/s)}G_{t}(\mu)(z)$$ belongs to   $L^{\infty}(\D, dA)$.
 
\item[(c)] The function  $$\tau(z)^{2(1-1/s)}\widehat{\mu_{\delta}}(z)$$ belongs to $L^{\infty}(\D, dA)$ for any sufficiently small $\delta>0$.
\end{enumerate}
\end{otherth}

Now we characterize $q$-Carleson measures when $0< q<p<\infty.$  

\begin{otherth}\label{thm:CMOP}
 Let $\mu$   be a finite positive Borel measure on $\D.$ Assume $0< q < p<\infty$  and $s=p/q.$  The following conditions are  all equivalent:
\begin{enumerate}
\item[(a)]  The measure $\mu$ is a $q$-Carleson measure for $A^p_\om$.

\item[(b)]   For any (or some ) $r>0$, we have 
$$\widehat{\mu_{r}} \in L^{p/(p-q)}(\D, dA).$$ 

\item[(c)] For any $t>0$, $$G_{t}(\mu)\in L^{p/(p-q)}(\D, dA)$$.
\end{enumerate}
\end{otherth}

\subsection{Vanishing Carleson measures}\label{sec3:2}

\begin{defn} Let $\mu$ be a positive measure on $\D$ and fix $0< p,q<  \infty.$ We say  that  $\mu$ is  a vanishing $q$-Carleson measure for $A^p_\om$ if the inclusion $I_{\mu}:A^p_{\om}\longrightarrow L^q_\om$ is compact, or equivalently, if $$\int_{\D}|f_n(z)|^q\,\om(z)^{q/2}\, d\mu(z)\to 0$$
whenever $f_n$ is bounded in $A^p_{\om}$  and  converges to zero uniformly on each compact subsets of $\D$.
\end{defn}

The following three theorems characterize vanishing $q$-Carleson measures for $A^p_\omega$ when $0<p\le \infty$ and $0<q<\infty$.

\begin{otherth}\label{thm:VCMPQ}
Given $\tau \in \mathcal{L^*},$  let $\mu$  be a finite positive Borel measure on $\D.$ Assume $0< p \le q<\infty,$ $s=p/q,$ $0< t< \infty.$ The following statements are  all equivalent:
\begin{enumerate}
\item[(a)]   $\mu$ is a  vanishing $q$-Carleson measure for $A^p_\om$.

\item[(b)]   $ \tau(z)^{2(1-1/s)}G_{t}(\mu)(z)\to 0 $ as $|z|\to 1^{-}$.

\item[(c)]  $\tau(z)^{2(1-1/s)}\widehat{\mu_{\delta}}(z)\to 0$ as $|z|\to 1^{-},$ for any small enough $\delta>0$. 
\end{enumerate}
\end{otherth}

\begin{otherth}\label{thm:CMOP2}
Given $\tau \in \mathcal{L^*},$  let $\mu$  be a finite positive Borel measure on $\D.$ Assume $0< q < \infty.$  The following conditions are  all equivalent:
\begin{enumerate}
\item[(a)]  $\mu$ is a $q$-Carleson measure for $A^{\infty}_\om.$
\item[(b)]  $\mu$ is a vanishing $q$-Carleson measure for $A^{\infty}_\om.$ 
\item[(c)] For any sufficiently small $\delta>0$ , we have  $$\widehat{\mu_{\delta}} \in L^{1}(\D, dA).$$
\item[(d)] For any $t>0$ , we have  $$G_t(\mu) \in L^{1}(\D, dA).$$
\end{enumerate}
\end{otherth}

\begin{thm}\label{thm:VCMQP}
Given $\tau \in \mathcal{L^*},$  let $\mu$  be a finite positive Borel measure on $\D.$ Assume that  $0< q < p<\infty.$ The following statements are  equivalent:
\begin{enumerate}
\item[(a)]   $\mu$ is a   $q$-Carleson measure for $A^p_\om.$ 
\item[(b)]   $\mu$ is a  vanishing $q$-Carleson measure for $A^p_\om$ .
\end{enumerate}
\end{thm}

\section{Embedding Theorems}\label{embedding}
In this section we establish embedding theorems of $S^p_\omega$ into $L^q(\D,d\mu)$ for  $0<p,q\le \infty$ and $\omega \in \mathcal{W}$, where $S^p_\omega$ are given in~\eqref{e:S^p} and~\eqref{e:S^infty}.
We start with the case $0<p\le q< \infty.$

\begin{lem} \label{lem1} 
Let $\omega \in \mathcal{W}$ and    $0<p\le q< \infty.$ Let $\mu$ be a positive Borel measure on $\D.$ Then  

\begin{enumerate}
\item[(a)]	$I_{\mu}: S^p_\omega\rightarrow L^q(\D,d\mu)$  is bounded if and only if for each $\delta>$ small enough,
\begin{equation}\label{eq1}
K_{\mu,\omega}(z)\:= \sup_{z\in \D}\frac{1}{\tau(z)^{2q/p}}\int_{D_{\delta}(z)} (1+\varphi'(\xi))^q\omega(\xi)^{-q/2}\,d\mu(\xi)<\infty.
\end{equation}

\item[(b)] $I_{\mu}: S^p_\omega\rightarrow L^q(\D,d\mu)$  is  compact if and only if 
\begin{equation}\label{eq2}
\lim_{|z|\to 1^-}\frac{1}{\tau(z)^{2q/p}}\int_{D_{\delta}(z)} (1+\varphi'(\xi))^q\omega(\xi)^{-q/2}\,d\mu(\xi)=0.
\end{equation}
\end{enumerate}
\end{lem}

\begin{proof}
Suppose first that the condition \eqref{eq1}  {holds.}  Then, by  Lemma  \ref{lem:Lcoverings} and \eqref{Eq-gamma}, we get 
\begin{align*}
	&\|f\|_{L^q(\mathbb{D},d\mu)}^q=\int_{\mathbb{D}} |f(z)|^q\, d\mu(z)\leq \sum_{k=0} ^{\infty}\, \int_{D_{\delta}(z_{k})} |f(z)|^q\, d\mu(z)\\
	&=\sum_{k=0} ^{\infty}\, \int_{D_{\delta}(z_{k})}  |f(z)|^q\, \frac{\omega(z)^{q/2}}{(1+{\varphi'}(z))^q}\,(1+{\varphi'}(z))^q\, \omega(z)^{-q/2}\,d\mu(z)\\
	&\lesssim \sum_{k=0} ^{\infty} \int_{D_{\delta}(z_{k})}\left(\frac{1}{\tau(z)^2}   \int_{D_{\delta}(z)} |f(s)|^p \frac{\omega(s)^{p/2}}{(1+{\varphi'}(s))^p}dA(s)\right)^{q/p}(1+{\varphi'}(z))^q \omega(z)^{-q/2}d\mu(z)\\
	&\lesssim   \sum_{k=0} ^{\infty}\left( \int_{D_{3\delta}(z_{k})} |f(s)|^p \frac{\omega(s)^{p/2}}{(1+{\varphi}^{'}(s))^p}\,dA(s)\right)^{q/p}\, \int_{D_{\delta}(z_{k})} \frac{(1+{\varphi'}(z))^q \omega(z)^{-q/2}}{\tau(z)^{2q/p}}d\mu(z),
\end{align*}
for  small enough $\delta>0$. By applying our assumption, we have
$$
\int_{\D}|f(z)|^q\, d\mu(z)
\lesssim K_{\mu,\omega} \sum_{k=0}^{\infty}\left(\int_{D_{3\delta}(z_k)}|f(s)|^p\,\frac{\omega(s)^{p/2}}{(1+\varphi'(s))^p}\,dA(s)\right)^{q/p}.
$$
Using a similar argument as in the proof of Theorem $1.1$ in \cite{O}, Minkowski's inequality and the finite multiplicity $N$ of the covering $\{D_{3\delta}(z_k)\}$, we get
$$
	\|f\|_{L^q(\mathbb{D},d\mu)}^q 
	\lesssim K_{{\mu,\omega}}  \left(\sum_{k=0} ^{\infty} \int_{D_{3\delta}(z_{k})} |f(s)|^p \frac{\omega(s)^{p/2}}{(1+{\varphi}^{'}(s))^p}\,dA(s)\right)^{q/p}
	\lesssim K_{\mu,\omega} \,N^{q/p}\, \|f\|^q_{S^p_\omega}.
$$
This  proves that the embedding $I_{\mu}: S^p_\omega\rightarrow L^q(\D,d\mu)$ is bounded with $\|I_\mu\|_{L^q(\mathbb{D},d\mu)}^q\leq K_{\mu,\omega}.$ 

Conversely, suppose that $I_{\mu}: S^p_\omega\rightarrow L^q(\D,d\mu)$  is bounded. Let $a\in \D$ with $|a|\geq \rho_0$ that is defined in Lemma \ref{Borichevlemma}. By Lemma \ref{BL3},   
$$
	|F'_{a,n}(z)|\, \omega(z)^{1/2} \asymp (1+\varphi'(z)),\ z\in D_\delta(a),
$$
(where $F_{a,n}$ is the test function  in Lemma \ref{Borichevlemma}), and so
$$
	\int_{D_{\delta}(a)}(1+\varphi'(z))^q\omega(z)^{-\frac{q}{2}}\,d\mu(z)
	\lesssim \int_{D_{\delta}(a)} |F'_{a,n}(z)|^q\,d\mu(z)
	\lesssim \int_{\D} |F'_{a,n}(z)|^q\,d\mu(z).
$$
Using our assumption, (a) of Lemma \ref{Borichevlemma}, and \eqref{littleW}, we obtain 
\begin{align*}
	\int_{D_{\delta}(a)}(1+\varphi'(z))^q\omega(z)^{-\frac{q}{2}}\,d\mu(z)
	&  \lesssim \|I_\mu\|^q\,\|F'_{a,n}\|^q_{S^p_\omega}\\
	&{\le \|I_\mu\|^q\,\|F_{a,n}\|^q_{A^p_\omega}}\asymp  \|I_\mu\|^q\,\tau(a)^{2q/p}.
\end{align*}
Then dividing both sides by $\tau(a)^{2q/p}$ gives
\begin{align*}
\frac{1}{\tau(a)^{2q/p}}\int_{D_{\delta}(a)}(1+\varphi'(z))^q\omega(z)^{-\frac{q}{2}}\,d\mu(z)\le \|I_\mu\|^q<\infty.
\end{align*}
and so
$$
	\sup_{a\in \mathbb{D}} \frac{1}{\tau(a)^{2q/p}} \int_{D_{\delta}(a)}  (1+{\varphi}^{'}(z))^q\, \omega(z)^{-q/2}\,d\mu(z) \le \|I_\mu\|^q<\infty,
$$
which means that $K_{\mu,\omega}\lesssim \|I_\mu\|^q$. 

To prove (b), suppose that $I_{\mu}: S^p_\omega\rightarrow L^q(\D,d\mu)$  is  compact. Consider the function 
$$
	f_{a,n}(z):=\frac{F_{a,n}(z)}{\tau(a)^{2q/p}},\quad \textrm{for}\,\,\, |a|\geq \rho_0.
$$ 
As in the proof of Theorem $1$ of \cite{PP1} and using Lemma  \ref{Borichevlemma}, we can show that   the   function $f_{a,n}$ is  bounded and converges  to zero uniformly on compact subsets of $\D$ when $|a|\to 1^-$. Therefore, by Lemma \ref{DERI}, $f'_{a,n}$ converges to zero uniformly on compact subsets of $\D$ as $|a|\to1^-$.
\begin{align*}
	\frac{1}{\tau(a)^{2q/p}}\int_{D_{\delta}(a)}(1+\varphi'(z))^q\omega(z)^{-\frac{q}{2}}\,d\mu(z)&\lesssim \int_{D_{\delta}(a)} |f'_{a,n}(z)|^q\,d\mu(z)\\
	&\le \int_{\D} |f'_{a,n}(z)|^q\,d\mu(z)=\|I_\mu f'_{a,n}\|_{L^q(\mu)}.
\end{align*}
Since $I_{\mu}$ is compact,
$$
	\lim_{|a|\to 1^{-}} \|f'_{a,n}\|_{L^q(\mu)}=0,
$$
and so 
$$
	\lim_{|a|\to 1^{-}} \frac{1}{\tau(a)^{2q/p}}\int_{D_{\delta}(a)}(1+\varphi'(z))^q\omega(z)^{-\frac{q}{2}}\,d\mu(z)=0.
$$
This shows that \eqref{eq2} holds.
 
Conversely, suppose that \eqref{eq2} holds. Let $\{f_n\}\subset S^p_\omega$ be a bounded sequence converging  to zero uniformly on compact subsets of $\D$ and  $\{z_k\}$ be a $(\delta,\tau)$-lattice.   To prove that $I_{\mu}$ is compact, it suffices to show that $\|f_n\|_{L^q(\mu)}\rightarrow 0.$ By the assumption, given any $\varepsilon >0,$ there exists $0< r_1<1$ with 
\begin{equation}
	\frac{1}{\tau(a)^{2q/p}}\int_{D_{\delta}(a)}(1+\varphi'(z))^q\omega(z)^{-\frac{q}{2}}\,d\mu(z)
	<\varepsilon, \quad r_1<|a|<1.
\end{equation}
Observe that there is $r_1<r_2<1$ such that if a point $z_j$ of the sequence $\{z_k\}$ belongs to $\{z\in\D: |z|\le r_1\},$ then $D_\delta(z_j)\subset \{z\in\D: |z|\le r_2\}.$ {Therefore, since $\{f_n\}$ converges to zero uniformly on compact subsets of $\D,$ }there exists an integer $n_0$ such that 
\begin{align*}
|f_n(z)|<\varepsilon,\quad \textrm{for}\,\,\, |z|\le r_2 \,\, \textrm{and}\,\,\, n\geq n_0.
\end{align*}
 {We split the integration of this function {into} two parts: the first integration is over ${|z|\le r_2}$ and the other integration is over ${|z|\geq r_2}$.}
On the one hand, 
\begin{equation}\label{eq4}
\begin{split}
\int_{|z|\le r_2}|f_n(z)|^q\,d\mu(z)<\varepsilon^q.
\end{split}
\end{equation}
On the other hand, by Lemma \ref{lem:Lcoverings} and Lemma \ref{DERI},  we obtain
\begin{align*}
	&\int_{|z|> r_2}|f_n(z)|^q\,d\mu(z)\le \sum_{{|z_k|>r_1}}\int_{D_\delta(z_k)}|f_n(z)|^q\,d\mu(z)\\
&\lesssim \sum_{|z_k|>r_1}\int_{D_\delta(z_k)}\left(\frac{1}{\tau(z_k)^2}\int_{D_{\delta}(z)}|f_n(s)|^p\,\frac{\omega(s)^{p/2}}{(1+\varphi'(s))^p}\,dA(s)\right)^{q/p}\, (1+\varphi'(z))^q\omega(z)^{-\frac{q}{2}}\,d\mu(z)\\&\lesssim  \sum_{k=0}^{\infty}\left(\int_{D_{3\delta}(z_k)}|f_n(s)|^p\,\frac{\omega(s)^{p/2}}{(1+\varphi'(s))^p}\,dA(s)\right)^{q/p}\int_{D_{\delta}(z_k)}\frac{(1+\varphi'(z))^q\omega(z)^{-\frac{q}{2}}}{\tau(z_k)^{2q/p}}\,d\mu(z)\\
&\lesssim \varepsilon \|f_n\|_{S^p_\omega}^q\,\sup_{|z_k|>r_1}\frac{1}{\tau(z_k)^{2q/p}}\int_{D_{\delta}(z_k)}(1+\varphi'(z))^q\omega(z)^{-\frac{q}{2}}\,d\mu(z)\lesssim \varepsilon \, \|f_n\|_{S^p_\omega}^q \lesssim \varepsilon.
\end{align*}
These together with \eqref{eq4} show that $I_{\mu}: S^p_\omega\rightarrow L^q(\D,d\mu)$ is compact.
\end{proof}

To characterize boundedness and compactness of $I_{\mu} : S^p_\omega \to L^q(\D,d\mu)$ with $0 < q < p < \infty$, consider the function $F_{\delta,\mu}(\varphi)$ defined by
 \begin{equation}\label{F-delta-mu}
 F_{\delta,\mu}(\varphi)(z):= \frac{1}{\tau(z)^{2}}\int_{D_{\delta}(z)} (1+\varphi'(\xi))^q\omega(\xi)^{-q/2}\,d\mu(\xi).
 \end{equation}
We use Luecking's approach in \cite{Lue1} based on   Khinchine's inequality. Recall that Rademacher functions $R_n$ {are defined by} 
$$
	R_{0}(t) =
	\begin{cases}
	1      & \quad \text{if } 1\le t-[t]< 1/2 \\
	-1 & \quad \text{if } 1/2\le t-[t]< 1;
\end{cases}
$$
$$
	R_{n}(t) =R_{0}(2^nt),\quad n\geq 1,
$$
where $[t]$ denotes the largest integer not exceeding $t.$

\begin{otherl}[Khinchine's inequality \cite{Lue1}]\label{Khin}
For $0< p < \infty,$ there exists a positive constant $C_p$ such that  
\[ 
C_p^{-1}\Big(\sum_{\substack{k = 1}}^{n}|\lambda_k|^2\Big)^{p/2}\le \int_{0}^{1}\Big|\sum_{\substack{k = 1}}^{n}\lambda_k R_k(t)\Big|^p dt \le C_p \Big(\sum_{\substack{k = 1}}^{n}|\lambda_k|^2\Big)^{p/2},
\]
for all $n\in \mathbb{N}$ and $\{\lambda_k\}^n_{k=1}\subset \C.$
\end{otherl}

\begin{lem} \label{lem2}  Let $\omega \in \mathcal{W}$ and    $0<q<p< \infty.$ Let $\mu$ be a finite positive Borel measure on $\D.$ Then,  the following statements are equivalent:
\begin{enumerate} 
\item[(a)] The operator $I_{\mu}: S^p_\omega\rightarrow L^q(\D,d\mu)$  is bounded.
\item[(b)] The operator $I_{\mu}: S^p_\omega\rightarrow L^q(\D,d\mu)$  is compact.
\item[(c)] The function 
\begin{equation}\label{eq3_1}
	F_{\delta,\mu}(\varphi)\in L^{p/(p-q)}(\D,dA).
\end{equation}
\end{enumerate}
\end{lem}

\begin{proof}
The implication (b) $\Rightarrow$ (a) is obvious. To prove that (a) implies (c), suppose that the operator $I_{\mu}: S^p_\omega\rightarrow L^q(\D,d\mu)$  is bounded. Let $\{z_k\}$ be a $(\delta,\tau)$-lattice on $\D$. Corresponding to each $\lambda=\{\lambda_m\}_m \in \ell^p, $ we consider 
$$
	f(z)= \sum_{|z_m|\geq \rho_0}\lambda_{m} f_{z_m,n}(z),
$$
where $f_{z_m,n}(z)=\frac{F_{z_m,n}(z)}{\tau(z_m)^{2/p}}$\,\, {and $0<\rho_0<1$} as in Lemma~\ref{Borichevlemma}. By Proposition \ref{At1-pp} and \eqref{littleW},
$$
	\|f'\|_{S^p_{\om}}
	\lesssim \|f\|_{A^p_{\om}}
	\lesssim \|\lambda\|_{\ell^p}.
$$
Note that as an application of Khinchine’s inequality (Lemma \ref{Khin}), replace $\lambda_m$ with the Rademacher functions $R_m(t)\lambda_m$, and then integrate with respect to $t$ from 0 and 1, {which yields}
$$
	\left(\sum_{|z_m|\geq\rho_0}\Big|\lambda_m\Big|^2\,\Big|f'_{z_m,n}(z)\Big|^2 \,\right)^{q/2} 
	\lesssim \int_{0}^{1} \Big|  \sum_{|z_m|\geq\rho_0}R_m(t)\lambda_m\, f'_{z_m,n}(z)\Big|^q  \,dt
$$
and so
\begin{equation}\label{eq7}
\begin{split}
	&\int_{\D}\left(\sum_{|z_m|\geq\rho_0}|\lambda_m|^{2} |f'_{z_m,n}(z)|^{2} \, \omega(z)\,\right)^{q/2}\,d\mu(z)\\&\lesssim \int_{\D}\int_{0}^{1} \left|  \sum_{z_m:|z_m|\geq\rho_0} R_m(t)\lambda_m\,f'_{z_m,n}(z)\right|^q\, \omega(z)^{q/2}  dt\,d\mu(z)\\&= \int_{0}^{1} \int_{\D} \left|  \sum_{z_m:|z_m|\geq\rho_0} R_m(t)\lambda_m\,f'_{z_m,n}(z)\right|^q\, \omega(z)^{q/2}d\mu(z)\,dt\\&  \lesssim {\int_{0}^{1}  \|f'\|^q_{S^p_\omega}\,dt= \|f'\|^q_{S^p_\omega}\,\lesssim\,  \|f\|^q_{A^p_\omega}}
\lesssim \|\lambda\|^q_{\ell^p}.
\end{split}
\end{equation}
By Lemmas \ref{lem:Lcoverings} and \ref{BL3},
\begin{align*}
	\sum_{\substack{|z_m|\geq\rho_0}}\,\frac{|\lambda_m|^q}{\tau(z_m)^{2/p}}&\int_{D_{3\delta}(z_m)} (1+\varphi'(\xi))^q\omega(\xi)^{-\frac{q}{2}}\,d\mu(\xi)\\
&\lesssim \sum_{\substack{|z_m|\geq\rho_0}}{|\lambda_m|^q}\,\int_{D_{3\delta}(z_m)} |f'_{z_m,n}(\xi)|^q\,d\mu(\xi)\\
&= \int_{\D}\sum_{\substack{|z_m|\geq\rho_0}}{|\lambda_m|^q}\, |f'_{z_m,n}(\xi)|^q\chi_{D_{3\delta}(z_m)}(\xi)\,d\mu(\xi),
\end{align*}
where  $\chi_{D_{3\delta}(z_m)}(\xi)$  denotes the characteristic function of the set ${D_{3\delta}(z_m)}$. Now, by the fact that $\sum_{\substack{s}}^{\infty} z_m^k \le \left(\sum_{\substack{s}}^{\infty} z_m\right)^k , \, k\geq 1, z_m \geq 0$ for $q\geq 2$, we get
\begin{align*}
 	&\int_{\D}\sum_{\substack{|z_m|\geq\rho_0}}{|\lambda_m|^q}\, |f'_{z_m,n}(\xi)|^q\chi_{D_{3\delta}(z_m)}(\xi)\,d\mu(\xi)\\
	&=\int_{\D}\sum_{\substack{|z_m|\geq\rho_0}}\left({|\lambda_m|^2}\, |f'_{z_m,n}(\xi)|^2\chi_{D_{3\delta}(z_m)}(\xi)\right)^{q/2}d\mu(\xi)\\ 
	&\lesssim \int_{\D}\Big(\sum_{\substack{|z_m|\geq\rho_0}}|\lambda_m|^2\,|f'_{z_m,n}(\xi)|^2\Big)^{q/2}\,d\mu(\xi).
\end{align*}
For $q<2$,  by H\"older's inequality and Lemma \ref{lem:Lcoverings}, we get 
\begin{align*}
	& \int_{\D}\sum_{\substack{|z_m|\geq\rho_0}}|\lambda_m|^q\, |f'_{z_m,n}(\xi)|^q\chi_{D_{3\delta}(z_m)}(\xi)\,d\mu(\xi)\\&\leq  \int_{\D}\left(\sum_{\substack{|z_m|\geq\rho_0}}|\lambda_m|^2\, |f'_{z_m,n}(\xi)|^2\right)^{q/2} \left(\sum_{\substack{|z_m|\geq\rho_0}} \chi_{D_{3\delta}(z_m)}(\xi)\right)^{1-\frac{q}{2}}d\mu(z) \\
&\lesssim N^{1-\frac{q}{2}}\int_{\D}\Big(\sum_{\substack{|z_m|\geq\rho_0}}|\lambda_m|^2\,|f'_{z_m,n}(\xi)|^2 \,\Big)^{q/2}\,d\mu(z).
\end{align*}
Therefore, for $q<2$ and $q\geq 2$, we have
\begin{align*}
	&\sum_{\substack{|z_m|\geq\rho_0}}\,\frac{|\lambda_m|^q}{\tau(z_m)^{2/p}}\int_{D_{3\delta}(z_m)} (1+\varphi'(\xi))^q\omega(\xi)^{-\frac{q}{2}}\,d\mu(\xi)\\
	&\lesssim \int_{\D}\sum_{\substack{|z_m|\geq\rho_0}}|\lambda_m|^q\, |f'_{z_m,n}(\xi)|^q\chi_{D_{3\delta}(z_m)}(\xi)\,d\mu(\xi)\\
	&\lesssim \max(1,N^{1-\frac{q}{2}})\int_{\D}\Big(\sum_{\substack{|z_m|\geq\rho_0}}|\lambda_m|^2\,|f'_{z_m,n}(\xi)|^2 \, \Big)^{q/2}\,d\mu(z).
\end{align*}
By applying \eqref{eq7}, we have
\begin{align*}
\sum_{\substack{|z_m|\geq\rho_0}}\frac{|\lambda_m|^q}{\tau(z_m)^{2/p}}\int_{D_{3\delta}(z_m)} (1+\varphi'(\xi))^q\omega(\xi)^{-\frac{q}{2}}\,d\mu(\xi)
\lesssim \|\lambda\|^q_{\ell^p}.
\end{align*}
Thus, taking $|b_m| = |\lambda_m|^q \in \ell^{p/q}$ and using  the duality $(\ell^{p})^{*}=\ell^{q}$, we see that the sequence 
$$
	\Big\{\frac{1}{\tau(z_m)^{2/p}}\int_{D_{3\delta}(z_m)} (1+\varphi'(\xi))^q\omega(\xi)^{-\frac{q}{2}}\,d\mu(\xi)\Big\}_m \in \ell^{^{p/(p-q)}}.$$
Observe that there is $\rho_0<r_1<1$ such that if a point $z_j$ of the sequence $\{z_k\}$ belongs to $\{z\in\D: |z|\le\rho_0 \},$ then $D_\delta(z_j)\subset \{z\in\D: |z|\le r_1\}.$
Thus, by Lemma \ref{lem:Lcoverings} and \eqref{eqn:asymptau}, we get
\begin{align*}
	&\int_{|z|\geq r_1}\left(\frac{1}{\tau(z)^{2}}\int_{D_{\delta}(z)} (1+\varphi'(\xi))^q\omega(\xi)^{-q/2}\,d\mu(\xi)\right)^{p/(p-q)}dA(z)\\
	&\lesssim \sum_{\substack{|z_m|\geq\rho_0}}\int_{D_{\delta}(z_m)}\left(\frac{1}{\tau(z)^{2}}\int_{D_{\delta}(z)} (1+\varphi'(\xi))^q\omega(\xi)^{-q/2}\,d\mu(\xi)\right)^{p/(p-q)}dA(z)\\
	&\lesssim \sum_{\substack{|z_m|\geq\rho_0}}\left(\frac{1}{\tau(z_m)^{2/p}}\int_{D_{3\delta}(z_m)} (1+\varphi'(\xi))^q\omega(\xi)^{-q/2}\,d\mu(\xi)\right)^{p/(p-q)}<\infty.
\end{align*}
Therefore, since 
$$
	\int_{|z|\leq r_1}\left(\frac{1}{\tau(z)^2} \int_{D_\delta(z)} (1+\varphi'(s))^q\omega(s)^{-\frac{q}{2}}d\mu(s)\right)^{p/p-q}d\mu(z) < \infty,
$$
we obtain 
\begin{align*}
	\int_{\D} &F_{\delta,\mu}(\varphi)(z)^{p/(p-q)}\,dA(z)
	=\int_{\D}\left(\frac{1}{\tau(z)^2}\int_{D_\delta(z)}(1+\varphi'(\xi))^q\omega(\xi)^{-\frac{q}{2}}\,d\mu(\xi)\right)^{p/(p-q)}\,dA(z)\\
	&\lesssim \int_{|z|<r_1}\left(\frac{1}{\tau(z)^2} \int_{D_\delta(z)} (1+\varphi'(s))^q\omega(s)^{-\frac{q}{2}}d\mu(s)\right)^{p/p-q}dA(z) \\
	&+ \int_{|z|>r_1}\left(\frac{1}{\tau(z)^2} \int_{D_\delta(z)} (1+\varphi'(s))^q\omega(s)^{-\frac{q}{2}}d\mu(s)\right)^{p/p-q}dA(z)<\infty.
\end{align*}
This proves the desired result.

Finally, it remains to prove that (c) implies (b). Suppose that \eqref{eq3_1} holds and let  $\{f_n\}$ be a bounded sequence of functions belonging to $S^p_\omega$ that converges uniformly to zero on compact subsets of $\D$. Since the function $\tau$ is decreasing  and converges to zero as $|z|\to 1$, there is  $r'>0$ such that 
\begin{equation}\label{eq-eq}
	D_{\delta/2}(z)\subset \Big\{\xi\in \D: |\xi|>r/2\Big\}, \quad \textrm{ if}\,\,\,  |z|>r>r'.
\end{equation}
On the other hand, it also follows from  \eqref{Eq-gamma} that 
$$
|f_n(z)|^q\,\frac{\omega(z)^{q/2}}{(1+\varphi'(z))^q} \lesssim \frac{1}{\tau(z)^2}\int_{D_{\delta}(z)}|f_n(s)|^q\,\frac{\omega(s)^{q/2}}{(1+\varphi'(s))^q}\,dA(s).
$$
Integrate with respect to $d\mu$, and use \eqref{eq-eq} and \eqref{eqn:asymptau} to obtain 
\begin{equation}\label{eq-zeq1}
\begin{aligned}
	&\int_{|z|\geq r} |f_n(\xi)|^q\,d\mu(\xi)\\
	&\lesssim \int_{|\xi|\geq r/2}|f_n(\xi)|^q\,\frac{\omega(\xi)^{q/2}}{(1+\varphi'(\xi))^q}\left( \frac{1}{\tau(\xi)^{2}}\int_{D_{\delta}(\xi)} (1+\varphi'(z))^q\,\omega(z)^{-q/2}\,d\mu(z)\right)dA(\xi)
\end{aligned}
\end{equation}
By (c), for each $\varepsilon >0$, there is an $r_0>r'$ such that 
$$
	\int_{|\xi|\geq r_0/2}\left( \frac{1}{\tau(\xi)^{2}}\int_{D_{\delta}(\xi)} (1+\varphi'(z))^q\,\omega(z)^{-q/2}\,d\mu(z)\right)^{p/(p-q)}dA(\xi)< \varepsilon^{p/(p-q)}.
$$
Combining this with H\"older's inequality, we have
\begin{equation}
\begin{aligned}
	&\int_{|z|\geq r_0} |f_n(\xi)|^q\,d\mu(\xi)\\
	&\lesssim\|f_n\|^{q} _{S^p_\omega}\left( \int_{|\xi|\geq r_0/2}\left( \frac{1}{\tau(\xi)^{2}}\int_{D_{\delta}(\xi)} (1+\varphi'(z))^q\,\omega(z)^{-q/2}\,d\mu(z)\right)^{p/(p-q)}dA(\xi)\right)^{(p-q)/p}
	\lesssim\varepsilon.
\end{aligned}
\end{equation}
This together with the fact that 
$$
	\lim_{n\to \infty}\int_{|z|\le r_0} |f_n(\xi)|^q\,d\mu(\xi)=0
$$
gives $\lim_{n\to \infty}\|f_n\| _{L^q(\mu)}=0$, which completes the proof.
\end{proof}

We finish this section with the case $0 < q < \infty$ and $p =\infty$.

\begin{lem} \label{lem3} 
Let $\omega \in \mathcal{W}$, $0<q<\infty$, and $\mu$ be a finite positive Borel measure on $\D$. Then the following statements are equivalent;
\begin{enumerate} 
\item[(a)] The operator $I_{\mu}: S^\infty_\omega\rightarrow L^q(\D,d\mu)$  is bounded.
\item[(b)] The operator $I_{\mu}: S^\infty_\omega\rightarrow L^q(\D,d\mu)$  is compact.
\item[(c)] The function 
\begin{equation}
	F_{\delta,\mu}(\varphi) \in L^{1}(\D,dA).
\end{equation}
\end{enumerate}
\end{lem}
\begin{proof}
Suppose first that the operator $I_{\mu}: S^\infty_\omega\rightarrow L^q(\D,d\mu)$  is bounded. Let  $\{z_m\}_m$ be a $(\delta,\tau)$-lattice on $\D.$  Corresponding to each $\lambda=\{\lambda_m\}_m \in \ell^\infty, $ we consider again
$$
	f(z)= \sum_{|z_m|\geq \rho_0}\lambda_{m} F_{z_m,n}(z),
$$
where $F_{z_m,n}(z)$ is  in Lemma \ref{Borichevlemma}. By Proposition \ref{At1-pp} and (\ref{littleW1}) , we have $$\|f'\|_{S^\infty_\omega}\leq\|f\|_{A^\infty_{\om}}\lesssim \|\lambda\|_{\ell^\infty}.$$ By our assumption, we get 
$$
	\int_{\D}\Big|\sum_{|z_m|\geq\rho_0}\lambda_m F'_{z_m,n}(z)\Big|^q \, \om(z)^{q/2}\, d\mu(z)
	\lesssim\|\lambda\|^q_{\ell^\infty}
$$
and so
\begin{equation*}
	\int_{\D}\left(\sum_{\substack{|z_m|\geq\rho_0}}|\lambda_m|^2 |F'_{z_m,n}(z)|^2\, \om(z)\right)^{q/2}\, d\mu(z)\lesssim\|\lambda\|^q_{\ell^\infty}.
\end{equation*}
This together with Lemma \ref{lem:Lcoverings}, Lemma \ref{BL3} and  H\"older's inequality  imply that 
\begin{equation*}
\begin{aligned}
	\sum_{\substack{|z_m|\geq\rho_0}}|\lambda_m|^q\,&\int_{D_{3\delta}(z_m)} (1+\varphi'(\xi))^q\omega(\xi)^{-\frac{q}{2}}\,d\mu(\xi)\\
	&\lesssim \max(1,N^{1-\frac{q}{2}})\int_{\D}\Big(\sum_{\substack{|z_m|\geq\rho_0}}|\lambda_m|^2\,|F'_{z_m,n}(\xi)|^2 \, \omega(\xi)\Big)^{q/2}\,d\mu(z)
	\lesssim\|\lambda\|^q_{\ell^\infty}.
\end{aligned}
\end{equation*}
Then, taking  $|\lambda_m| = 1$ gives  
\begin{equation}\label{eq:rad1}
\begin{aligned}
	\sum_{\substack{|z_m|\geq\rho_0}}\,\int_{D_{3\delta}(z_m)} (1+\varphi'(\xi))^q\omega(\xi)^{-\frac{q}{2}}\,d\mu(\xi)
\lesssim 1.
\end{aligned}
\end{equation}
As in the previous proof, by Lemma~\ref{lem:Lcoverings}, \eqref{eqn:asymptau} and \eqref{eq:rad1}, we get
\begin{equation*}
\begin{aligned}
	&\int_{|z|\geq r_1}\left(\frac{1}{\tau(z)^{2}}\int_{D_{\delta}(z)} (1+\varphi'(\xi))^q\omega(\xi)^{-q/2}\,d\mu(\xi)\right)dA(z)\\
	&\lesssim \sum_{\substack{|z_m|\geq\rho_0}}\int_{D_{\delta}(z_m)}\left(\frac{1}{\tau(z)^{2}}\int_{D_{\delta}(z)} (1+\varphi'(\xi))^q\omega(\xi)^{-q/2}\,d\mu(\xi)\right)dA(z)\\
	&\lesssim \sum_{\substack{|z_m|\geq\rho_0}}\int_{D_{3\delta}(z_m)} (1+\varphi'(\xi))^q\omega(\xi)^{-q/2}\,d\mu(\xi)<\infty.
\end{aligned}
\end{equation*}
Combining this with the fact that 
$$
\int_{|z|\le r_1}\left(\frac{1}{\tau(z)^{2}}\int_{D_{\delta}(z)} (1+\varphi'(\xi))^q\omega(\xi)^{-q/2}\,d\mu(\xi)\right)dA(z)<\infty,
$$
we have the desired result---see \eqref{F-delta-mu}.

It remains to show that (c) implies (b). Let   $\{f_n\}$ be a bounded sequence of functions in $S^\infty_\omega$ converging uniformly to zero on compact subsets of $\D.$  Since the function $\tau(z)$ is decreasing  and converges  to zero as $|z|\to 1,$ there is $r'>0$ such that 
\begin{equation}\label{eq-zeq}
D_{\delta/2}(z)\subset \Big\{\xi\in \D: |\xi|>r/2\Big\}, \quad \textrm{ if}\,\,\,  |z|>r>r'.
\end{equation}
On the other hand, it follows from  \eqref{Eq-gamma} that 
$$
|f_n(z)|^q\,\frac{\omega(z)^{q/2}}{(1+\varphi'(z))^q} \lesssim \frac{1}{\tau(z)^2}\int_{D_{\delta}(z)}|f_n(s)|^q\,\frac{\omega(s)^{q/2}}{(1+\varphi'(s))^q}\,dA(s).
$$
Integrate with respect to $d\mu,$ and use \eqref{eq-zeq},  \eqref{eqn:asymptau}, and \eqref{Eq-gamma},  to obtain 
\begin{equation}
\begin{aligned}
	&\int_{|z|\geq r} |f_n(\xi)|^q\,d\mu(\xi)\\
	&\lesssim \int_{|\xi|\geq r/2}|f_n(\xi)|^q\,\frac{\omega(\xi)^{q/2}}{(1+\varphi'(\xi))^q}\left( \frac{1}{\tau(\xi)^{2}}\int_{D_{\delta}(\xi)} (1+\varphi'(z))^q\,\omega(z)^{-q/2}\,d\mu(z)\right)dA(\xi)\\
	&\lesssim\|f_n\| ^q_{S^\infty_\omega}\int_{|\xi|\geq r/2} \frac{1}{\tau(\xi)^{2}}\int_{D_{\delta}(\xi)} (1+\varphi'(z))^q\,\omega(z)^{-q/2}\,d\mu(z)dA(\xi).
\end{aligned}
\end{equation}
Now the rest follows as in the previous proof.
\end{proof}

\section{Proofs of Theorems~\ref{thm1} and \ref{thmV1}}\label{proofs1}

\subsection{Proof of Theorem~\ref{thm1} (A)} Let $0<p\leq q<\infty$. By \eqref{littleW}, 
\begin{align*}
	\|GI_{(\phi,g)} f(z)\|_{A_\omega^q}^q&\asymp \int_{\mathbb{D}} | f'(\phi(z))|^q\, |g(z)|^q\, \frac{\omega(z)^{q/2}}{(1+{\varphi}^{'}(z))^q}\,dA(z)\\
	&=  \int_{\mathbb{D}} | f'(z)|^q\, d\mu_{\phi,\omega,q}(z)=\|f'\|^q_{L^q(\mathbb{D},d\mu_{\phi,\omega,g})}.
\end{align*}
Therefore, $GI_{(\phi,g)}: A^p_\omega\rightarrow A^q_\omega$ is bounded if and only if $I_{\mu_{\phi,\omega,g}}: S^p_\omega\rightarrow L^q(\mu_{\phi,\omega,g})$ is bounded. Using (a) of Lemma \ref{lem1},  this is equivalent to 
$$
	\sup_{z\in \mathbb{D}}\,\frac{1}{\tau(z)^{2q/p}}\int_{D_\delta(z)} (1+{\varphi}^{'}(\xi))^q\, \omega(\xi)^{-q/2}\,d\mu_{\phi,\omega,g}(\xi)<\infty,
$$
which, by Theorem \ref{thm:CMPQ}, is equivalent to 
$$
	\sup_{z\in \D}\tau(z)^{2(1-q/p)}\int_{\D} |k_{q,z}(\xi)|^q\omega(\xi)^{q/2}\,d\nu_{\phi,\omega,g}(\xi)<\infty.
$$
Now, by $(a)$ of Lemma \ref{lem:RK-PE1},  we get 
\begin{equation}
\begin{aligned}\label{maincond1}
	\tau(z)^{2(1-q/p)}\int_{\mathbb{D}} |k_{q,z}(\xi)|^q\, \omega(\xi)^{q/2} \,dv_{\phi,\omega,q}(\xi) 
	&\asymp \int_{\mathbb{D}} |k_{p,z}(\xi)|^q\, \omega(\xi)^{q/2} \,dv_{\phi,\omega,q}(\xi)\\&=  \int_{\mathbb{D}} |k_{p,z}(\xi)|^q\, (1+{\varphi}^{'}(\phi(\xi))^q\,d\mu_{\phi,\omega,q}(\xi)\\
	&=   \int_{\mathbb{D}} |k_{p,z}(\xi)|^q\, |g(z)|^q\,  \frac{(1+{\varphi}^{'}(\phi(\xi)))^q}{(1+{\varphi}^{'}(\xi))^q}\,\omega(z)^{q/2}\,dA(\xi)\\&=GB^{\phi}_{1,p,q}.
\end{aligned}
\end{equation}
Thus, $GI_{\phi, g}$ is bounded if and only if $GB^{\phi}_{1,p,q} (g(z)) \in L^{\infty}(\mathbb{D},dA)$. Compactness can be proved similarly using (b) of Lemma~\ref{lem1}.

\subsection{Proof of Theorem~\ref{thm1} (B)} {\bf Boundedness.}
Let $0<p<q=\infty$ and suppose first that  \eqref{eq-G2} holds. {Then by \eqref{littleW1} and our assumption,} we have 
\begin{align*}
	\|GI_{\phi,g} f\|_{A_{\omega}^{\infty}}
	&\asymp   \sup_{z\in\mathbb{D}} |f'(\phi(z))| \,|g(z)|\,\frac{\omega(z)^{1/2}}{(1+{\varphi'}(z))}\\
	&\leq  \sup_{z\in \D} M_{g,\phi,\om}(z)\,\sup_{z\in \D}\frac{|f'(\phi(z))|\,\omega(\phi(z))^{\frac{1}{2}}}{(1+\varphi'(\phi(z)))}\,\Delta\varphi(\phi(z))^{-1/p}\\
	&\le \sup_{z\in \D} M_{g,\phi,\om}(z)\,\sup_{z\in \D}\frac{|f'(\phi(z))|\,\omega(\phi(z))^{\frac{1}{2}}}{(1+\varphi'(\phi(z)))}\,\tau(\phi(z))^{2/p}. 
\end{align*}
Therefore, by \eqref{Eq-gamma},
\begin{align*}
	\|GI_{(\phi,g)}f\|_{A^\infty_{\omega}}
	&\lesssim \sup_{z\in \D}\left(\int_{D_{\delta}(\phi(z))}\frac{|f'(\xi)|^p\omega(\xi)^{\frac{p}{2}}}{(1+\varphi'(\xi))^p}\, dA(\xi)\right)^{1/p}\\
	&\le\left(\int_{\D}\frac{|f'(\xi)|^p\omega(\xi)^{\frac{p}{2}}}{(1+\varphi'(\xi))^p}\, dA(\xi)\right)^{1/p}
{\lesssim}\, \, \|f\|_{A^p_{\omega}},
\end{align*}
which implies that $GI_{(\phi,g)}:A^{p}_{\omega} \to A^{q}_{\omega}$ is is bounded.

Conversely, suppose that the operator $GI_{(\phi,g)}: A^p_\omega\rightarrow A^\infty_\omega$ is bounded. Choose $\xi\in\D$ so that $|\phi(\xi)|>\rho_0$, and consider the function $f_{\phi(\xi),n,p}$ given by  
$$
	f_{\phi(\xi),n,p}:=\frac{F_{\phi(\xi),n,p}}{\tau(\phi(\xi))^{2/p}},  
$$
where $ F_{\phi(\xi),n,p}$ is the test function in Lemma \ref{Borichevlemma}. Notice that $f_{\phi(\xi),n,p}$ is in $A^p_{\omega}$ and $\|f_{\phi(\xi),n,p}\|_{A^p_{\omega}}\asymp 1$. By our assumption, we get 
\begin{align*}
	\|GI_{(\phi,g)}(f_{\phi(\xi),n,p})\|_{A^\infty_{\om}}&\geq  \sup_{z\in \D}\frac{|f'_{\phi(\xi),n,p}(z)||g(z)|}{(1+\varphi'(z))}\,\omega(z)^{\frac{1}{2}}\\
	&\geq\sup_{z\in \D}\frac{|F'_{\phi(\xi),n,p}(z)||g(z)|}{\tau(\phi(\xi))^{2/p}(1+\varphi'(z))}\,\omega(z)^{\frac{1}{2}}\\
	&\geq\sup_{\xi \in \D}\frac{|F'_{\phi(\xi),n,p}(\phi(\xi))||g(\xi)|}{\tau(\phi(\xi))^{2/p}(1+\varphi'(\xi))}\,\omega(\xi)^{\frac{1}{2}}.
\end{align*}
By Lemma \ref{BL3},
$$
	|F'_{\phi(\xi),n,p}(z)|\, \omega(z)^{1/2}\asymp (1+\varphi'(z)),\ z\in D_\delta(\phi(\xi)),
$$
and so
\begin{equation}\label{eq-M}
\begin{split}
	\infty >\|GI_{(\phi,g)}(f_{\phi(\xi),n,p})\|_{A^\infty_{\om}}
	&\geq|g(\xi)|\,\frac{(1+\varphi'(\phi(\xi)))}{(1+\varphi'(\xi))}\,\frac{\omega(\xi)^{\frac{1}{2}}}{\omega(\phi(\xi))^{\frac{1}{2}}}\,\tau(\phi(\xi))^{-2/p}\\
	&\asymp |g(\xi)|\,\frac{(1+\varphi'(\phi(\xi)))}{(1+\varphi'(\xi))}\,\frac{\omega(\xi)^{\frac{1}{2}}}{\omega(\phi(\xi))^{\frac{1}{2}}}\,\Delta\varphi(\phi(\xi))^{1/p}\\&=M_{g,\phi,\om}(\xi).
\end{split}
\end{equation}
On the other hand, by taking $f(z)=z$ and using the boundedness of the operator $GI_{(\phi,g)}: A^p_\omega\rightarrow A^\infty_\omega$, we obtain
$$
	\|GI_{(\phi,g)}f\|_{A^\infty_{\om}}
	=\sup_{z\in \D}\frac{|g(z)|}{(1+\varphi'(z))}\,\,\omega(z)^{\frac{1}{2}}\lesssim \|f\|_{A^p_{\om}} <\infty.
$$
Therefore, in the case of $|\phi(\xi)|\le\rho_0, \, \xi \in \D,$ we have 
\begin{align*}
	|M_{g,\phi,\om}(\xi)|
	&=|g(\xi)|\,\frac{(1+\varphi'(\phi(\xi)))}{(1+\varphi'(\xi))}\,\frac{\omega(\xi)^{\frac{1}{2}}}{\omega(\phi(\xi))^{\frac{1}{2}}}\,\Delta\varphi(\phi(\xi))^{1/p}\\
	&\asymp |g(\xi)|\,\frac{(1+\varphi'(\phi(\xi)))}{(1+\varphi'(\xi))}\,\frac{\omega(\xi)^{\frac{1}{2}}}{\omega(\phi(\xi))^{\frac{1}{2}}}\,\tau(\phi(\xi))^{-2/p}\\
	&\le C_1 \,\frac{|g(\xi)|}{(1+\varphi'(\xi))}\,\,\omega(\xi)^{\frac{1}{2}}<\infty,
\end{align*}
where
$$
	C_1 = \sup_{|\phi(\xi)|\le \rho_0}\Big\{(1+\varphi'(\phi(\xi)))\omega(\phi(\xi))^{\frac{-1}{2}} \,\tau(\phi(\xi))^{-2/p}\Big\}<\infty.
$$
Combining this with \eqref{eq-M} completes the proof of boundedness.

{\bf Compactness.} Suppose now that $GI_{(\phi,g)}: A^p_\omega\rightarrow A^\infty_\omega$ is compact. Then, since $f_{\phi(\xi),n,p}$ converges to zero uniformly on compact subsets of $\D$ as $|\phi(\xi)|\to 1$ (see Lemma 3.1 of \cite{PP1}), it follows that
$$
	\|GI_{(\phi,g)}(f_{\phi(\xi),n,p})\|_{A^\infty_{\om}}\rightarrow 0
$$
as $|\phi(\xi)|\to 1$. Thus, by \eqref{eq-M},
\begin{equation*}
	0=\lim_{|\phi(\xi)|\to1^-}\|GI_{(\phi,g)}(f_{\phi(\xi),n,p})\|_{A^\infty_{\om}} \gtrsim\lim_{|\phi(\xi)|\to1^-} M_{g,\phi,\om}(\xi).
\end{equation*}

To prove the converse, let $\{f_n\}$ be a bounded sequence of functions in $A^p_\omega$ converging uniformly to zero on compact subsets of $\D$. Since \eqref{eq-G2} holds, for each $\varepsilon >0$, there exists an $r_0>0$ such that 
$$
	M_{g,\phi,\om}(\xi)
	=|g(\xi)|\,\frac{1+\varphi'(\phi(\xi)}{1+\varphi'(\xi)}\frac{\omega(\xi)^{1/2}}{\omega(\phi(\xi))^{1/2}}\Delta\varphi(\phi(\xi))^{1/p}
	<\varepsilon,
$$
whenever $|\phi(\xi)|>r_0.$ In addition, by \eqref{Eq-gamma}, 
\begin{equation}\label{eq-ext1}
\begin{split}
	&\frac{|f_n'(\phi(\xi))||g(\xi)|}{1+\varphi'(\xi)}\,\,\omega(\xi)^{\frac{1}{2}}\\
	&\quad\lesssim \Big(\frac{1}{\tau(z)^2}\int_{D_{\delta}(z)} \frac{|f_n'(\phi(s))|^p}{(1+\varphi'(\phi(s)))^p}\,\,\omega(\phi(s))^{\frac{p}{2}}\,\Delta\varphi(\phi(\xi))\,dA(s)\Big)^{1/p}   M_{g,\phi,\om}(\xi) \\
	&\quad\lesssim \|f_n\|_{A^p_\omega} M_{g,\phi,\om}(\xi) < \varepsilon,
\end{split}
\end{equation}
whenever $|\phi(\xi)|>r_0$.

For $|\phi(\xi)| \ge r_0$, we have
\begin{equation*}
	\sup_{|\phi(\xi)|\le r_0} \frac{|f_n'(\phi(\xi))||g(\xi)|}{1+\varphi'(\xi)}\,\,\omega(\xi)^{\frac{1}{2}}
	\lesssim \sup_{|\phi(\xi)|\le r_0} |f_n'(\phi(\xi))|\rightarrow 0,
\end{equation*}
as $n\to \infty$ because the sequence of functions $f'_n$ also converges uniformly to zero on compact subsets of $\D$ (see Lemma \ref{DERI}). This together with \eqref{eq-ext1} yields 
$$
	\|GI_{(\phi,g)}(f_n)\|_{A^\infty_{\om}}
	{\asymp}\sup_{\xi\in \D}\frac{|f_n'(\phi(z))||g(\xi)|}{1+\varphi'(\xi)}\,\,
	\omega(\xi)^{\frac{1}{2}}\rightarrow 0,\quad n\to \infty,
$$
which shows the compactness of the operator $GI_{(\phi,g)}: A^p_\omega\rightarrow A^\infty_\omega$.

\subsection{Proof of Theorem~\ref{thm1} (C)} {\bf Boundedness.}
Let $p=q=\infty$ and suppose that $\eqref{eq-G3-NI}$ holds. Using \eqref{littleW1}, we get 
\begin{align*}
	\|GI_{(\phi,g)}f\|_{A^\infty_{\om}}
	&\asymp\sup_{z\in \D} \frac{|f'(\phi(z))||g(z)|}{(1+\varphi'(z))}\,\,\omega(z)^{\frac{1}{2}}\\
	&\le \sup_{z\in \D} N_{g,\phi,\omega}(z)\,\sup_{z\in \D}\frac{|f'(\phi(z))|\omega(\phi(z))^{\frac{1}{2}}}{(1+\varphi'(\phi(z)))}\\
	&\le \sup_{z\in \D} N_{g,\phi,\omega}(z)\,\sup_{z\in \D}\frac{|f'(z)|\omega(z)^{\frac{1}{2}}}{(1+\varphi'(z))}\lesssim  \|f\|_{A^\infty_\omega}
\end{align*}
which shows that $GI_{(\phi, g)}$ is bounded.

Conversely, suppose that $GI_{(\phi,g)}: A^\infty_\omega\rightarrow A^\infty_\omega$ is bounded. Let $\xi\in\D$ be such that $|\phi(\xi)|>\rho_0$. Then $F_{\phi(\xi),n,p} \in A^\infty_{\om}$ and $\|F_{\phi(\xi),n,p}\|_{A^\infty_{\om}}\asymp 1$, and hence
\begin{equation*}
	\infty >\|GI_{(\phi,g)}(F_{\phi(\xi),n,p})\|_{A^\infty_{\om}}
	=\sup_{z\in \D}\frac{|F'_{\phi(\xi),n,p}(z)||g(z)|}{(1+\varphi'(z))}\,\omega(z)^{\frac{1}{2}}\geq\frac{|F'_{\phi(\xi),n,p}(\phi(\xi))||g(\xi)|}{(1+\varphi'(\xi))}\,\omega(\xi)^{\frac{1}{2}}.
\end{equation*}
By Lemma \ref{BL3},
$$
	|F'_{\phi(\xi),n,p}(z)| \, \omega(z)^{1/2}\asymp (1+\varphi'(z)),\ z\in D_\delta(\phi(\xi)),
$$
so
\begin{equation}\label{eq-M1}
	\infty >\|GI_{(\phi,g)}(F_{\phi(\xi),n,p})\|_{A^\infty_{\om}}
	\asymp |g(\xi)|\,\frac{(1+\varphi'(\phi(\xi)))}{(1+\varphi'(\xi))}\,\frac{\omega(\xi)^{\frac{1}{2}}}{\omega(\phi(\xi))^{\frac{1}{2}}}=NI_{g,\phi,\omega}(\xi).
\end{equation}

To deal with the case $|\phi(\xi)|\le\rho_0$, take $f(z)=z$ and use the boundedness of the operator $GI_{(\phi,g)}$ to obtain
\begin{equation}\label{eq-ext}
	\|GI_{(\phi,g)}f\|_{A^\infty_{\om}} =\sup_{z\in \D}\frac{|g(z)|}{(1+\varphi'(z))}\,\,\omega(z)^{\frac{1}{2}}\lesssim \|f\|_{A^\infty_{\om}} <\infty.
\end{equation}
Therefore, when $|\phi(\xi)|\le\rho_0, \, \xi \in \D,$ we have 
\begin{equation*}
	|g(\xi)|\,\frac{(1+\varphi'(\phi(\xi)))}{(1+\varphi'(\xi))}\,\frac{\omega(\xi)^{\frac{1}{2}}}{\omega(\phi(\xi))^{\frac{1}{2}}}
	\le C_2 \,\frac{|g(\xi)|}{(1+\varphi'(\xi))}\,\,\omega(\xi)^{\frac{1}{2}}<\infty,
\end{equation*}
where $$C_2 = \sup_{|\phi(\xi)|\le \rho_0}\Big\{(1+\varphi'(\phi(\xi)))\omega(\phi(\xi))^{\frac{-1}{2}} \Big\}<\infty.$$
Combining this with \eqref{eq-M1} completes the proof of boundedness. 

{\bf Compactness.} If $GI_{(\phi,g)}: A^\infty_\omega\rightarrow A^\infty_\omega$ is compact, then, using~\eqref{eq-M1} again, we get
$$
	\lim_{|\phi(\xi)|\to1^-} NI_{g,\phi,\om}(\xi) 
	\lesssim \lim_{|\phi(\xi)|\to1^-}\|GI_{(\phi,g)}(f_{\phi(\xi),n,p})\|_{A^\infty_{\om}} = 0.
$$

To prove the converse, let $\{f_n\}$ be a bounded sequence of functions in $A^\infty_\omega$ converging uniformly to zero on compact subsets of $\D.$ By assumption, for any $\varepsilon >0,$ there exists $r_0>0$ such that 
$$
	NI_{g,\phi,\om}(\xi)
	=|g(z)|\, \,\frac{(1+\varphi'(\phi(z))}{(1+\varphi'(z))}
	\frac{\omega(z)^{1/2}}{\omega(\phi(z))^{1/2}} <\varepsilon,
$$
whenever $|\phi(\xi)|>r_0.$  Notice that 
\begin{equation}\label{eq-et1}
\begin{split}
	\frac{|f_n'(\phi(\xi))||g(\xi)|}{1+\varphi'(\xi)}\,\,\omega(\xi)^{\frac{1}{2}}
	&\lesssim \|f_n\|_{A^\infty_{\om}}\,
|g(\xi)|\,\frac{(1+\varphi'(\phi(\xi))}{(1+\varphi'(\xi))}\frac{\omega(\xi)^{1/2}}{\omega(\phi(\xi))^{1/2}} \\
	&= \|f_n\|_{A^\infty_{\om}}\, N_{g,\phi,\om}(\xi) < \varepsilon,
\end{split}
\end{equation}
whenever $|\phi(\xi)|>r_0.$ The rest follows as in the proof of (B).

\subsection{Proof of Theorem~\ref{thm1} (D)} Let $0<q<p< \infty$ and suppose that $GI_{(\phi,g)}: A^p_\omega\rightarrow A^q_\omega$ is bounded. If $\{f_n\}\subset A^p_\omega$ is a bounded sequence converging  to zero uniformly on compact subsets of $\D$, then 
\begin{equation}\label{ppqwer}
	\|GI_{(\phi,g)}f_n\|^q_{A^q_{\om}}
	\asymp \int_{\D}\frac{|f'_n(\phi(z))|^q|g(z)|^q}{(1+\varphi'(z))^q}\,\,\omega(z)^{\frac{q}{2}}\, dA(z)=\|f'_n\|^q_{L^q(\mu_{\phi,\omega,g})},
\end{equation}
which goes to zero as $n\to\infty$ because of the compactness of  the embedding $I_{\mu_{\phi,\omega,g}}$.

We next prove that (a) and (c) are equivalent. By \eqref{ppqwer} and Lemma  \ref{lem2}, we get $GI_{(\phi,g)}: A^p_\omega\rightarrow A^q_\omega$ is bounded if and only if $I_{\mu_{\phi,\omega,g}}: S^p_\omega\rightarrow L^q(\mu_{\phi,\omega,g})$ is bounded if and only if $I_{\mu_{\phi,\omega,g}}: S^p_\omega\rightarrow L^q(\mu_{\phi,\omega,g})$ is compact if and only if the function 
$$
	F_{\delta,\mu_{\phi,\omega,g}}(\varphi)(z):= \frac{1}{\tau(z)^{2}}\int_{D_{\delta}(z)} (1+\varphi'(\xi))^q\omega(\xi)^{-q/2}\,d\mu_{\phi,\omega,g}(\xi)
$$
belongs to $ L^{p/(p-q)}(\D,dA).$ By Theorem \ref{thm:CMOP}, this is equivalent to 
$$
	\int_{\D}|k_{q,z}(\xi)|^q\, \omega(\xi)^{q/2}\,d\nu_{\phi,\omega,g}(\xi) \in  L^{p/(p-q)}(\D,dA),
$$
which is as well equivalent to $GB^\phi_{1,p,q}(g)(z)\in L^{p/(p-q)}(\D,d\lambda)$, where $d\lambda(z)=dA(z)/\tau(z)^2,$ because of 
\begin{align*}
	\int_{\D}  G_{q}(\nu^q_{\phi,\omega,g})^{p/p-q}dA(z)
	&= \int_{\D}\Big( \tau(z)^{2(1-q/p)}\, G_{q}(\nu^q_{\phi,\omega,g})\Big)^{\frac{p}{p-q}}d\lambda(z)\\& \asymp \int_{\D}\Big( \tau(z)^{2(1-q/p)}\, \int_{\D} |k_{p,z}(\xi)|^q\omega(\xi)^{q/2}\,d\nu_{\phi,\omega,g}(\xi)\Big)^{\frac{p}{p-q}}d\lambda(z)\\
	&=\int_{\D}\Big( \tau(z)^{2(1-q/p)}\, \int_{\D} |k_{p,z}(\xi)|^q\,(1+\varphi'(\xi))^q\,d\mu_{\phi,\omega,g}(\xi)\Big)^{\frac{p}{p-q}}d\lambda(z)\\
	&=\int_{\D} GB^\phi_{1,p,q}(g)(z)^{p/p-q}\,d\lambda(z).
\end{align*}
This completes the proof of (D) when $0<q<p<\infty$.

Suppose that $0<q<p=\infty$. If $GI_{(\phi,g)}: A^\infty_\omega\rightarrow A^q_\omega$ is bounded and $\{f_n\}\subset A^\infty_\omega$ is a bounded sequence converging to zero uniformly on compact subsets of $\D$, then 
\begin{equation}\label{ppqwwer}
	\|GI_{(\phi,g)}f_n\|^q_{A^q_{\om}}
	\asymp\int_{\D}\frac{|f'_n(\phi(z))|^q|g(z)|^q}{(1+\varphi'(z))^q}\,\,\omega(z)^{\frac{q}{2}}\, dA(z)=\|f'_n\|^q_{L^q(\mu_{\phi,\omega,g})} \to 0,
\end{equation}
where we used again the compactness of  the embedding $I_{\mu_{\phi,\omega,g}}$, and so $GI_{(\phi, g)}$ is compact.

It remains to prove that (a) and (c) are equivalent. By \eqref{ppqwwer} and Lemma  \ref{lem3}, we get
$GI_{(\phi,g)}: A^\infty_\omega\rightarrow A^q\omega$ is bounded if and only if $I_{\mu_{\phi,\omega,g}}: S^\infty_\omega\rightarrow L^q(\mu_{\phi,\omega,g})$ is bounded if and only if $I_{\mu_{\phi,\omega,g}}: S^\infty_\omega\rightarrow L^q(\mu_{\phi,\omega,g})$ is compact if and only if the function 
$$
	F_{\delta,\mu_{\phi,\omega,g}}(\varphi)(z):= \frac{1}{\tau(z)^{2q/p}}\int_{D_{\delta}(z)} (1+\varphi'(\xi))^q\omega(\xi)^{-q/2}\,d\mu_{\phi,\omega,g}(\xi)
$$
belongs to $ L^{1}(\D,dA).$ By Theorem \ref{thm:CMOP2}, this is equivalent to 
$$
	\int_{\D}|k_{q,z}(\xi)|^q\, \omega(\xi)^{-q/2}\,d\nu_{\phi,\omega,g}(\xi) \in  L^{1}(\D,dA),
$$
which is in turn equivalent to $GB^\phi_{1,p,q}(g)(z)\in L^{1}(\D,d\lambda),$ where $d\lambda(z)=dA(z)/\tau(z)^2,$ because of 
$$
	GB^\phi_{1,p,q}(g)(z)
	\asymp  \tau(z)^{2(1-q/p)}\int_{\D} |k_{q,z}(\xi)|^q\omega(\xi)^{q/2}\,d\nu^q_{\phi,\omega,g}(\xi).
$$

\subsection{Proof of Theorem~\ref{thmV1} (A)} {\bf Boundedness.} Let $0<p\leq q<\infty$. By \eqref{littleW},
\begin{equation}\label{eq-eqv}
	\|GV_{(\phi,g)}f\|^q_{A^q_{\om}}
	\asymp \int_{\D}\frac{|f(\phi(z))|^q|g(z)|^q}{(1+\varphi'(z))^q}\,\,\omega(z)^{\frac{q}{2}}\, dA(z)=\int_{\D}|f(z)|^q\,\omega(z)^{\frac{q}{2}}\,d\nu_{\phi,\omega,g}.
\end{equation}
Therefore, $GV_{(\phi,g)}: A^p_\omega\rightarrow A^q_\omega$ is bounded if and only if  the measure $\nu_{\phi,\omega,g}$ is a $q$-Carleson measure for $A_\omega^p$. According to Theorem \ref{thm:CMPQ}, this is equivalent to 
$$
	\sup_{z\in \D}\tau(z)^{2(1-q/p)}\int_{\D} |k_{q,z}(\xi)|^q\omega(\xi)^{q/2}\,d\nu_{\phi,\omega,g}(\xi)<\infty.
$$
Now, using $(a)$ of Lemma \ref{lem:RK-PE1}, we get 
\begin{align*}\label{maincond1}
	\tau(z)^{2(1-q/p)}\int_{\mathbb{D}} |k_{q,z}(\xi)|^q\, \omega(\xi)^{q/2} \,d\nu_{\phi,\omega,q}(\xi)
	&\asymp \int_{\mathbb{D}} |k_{p,z}(\xi)|^q\, \omega(\xi)^{q/2} \,d\nu_{\phi,\omega,q}(\xi)\\
	&= \int_{\mathbb{D}} |k_{p,z}(\phi(\xi))|^q \, \frac{ |g(z)|^q}{(1+{\varphi}^{'}(\xi))^q}\,\omega(\xi)^{q/2}dA(\xi)\\&=GB^{\phi}_{0,p,q}.
\end{align*}
Thus, $GV_{(\phi,g)}$ is bounded if and only if $GB^\phi_{0,p,q}(g)\in L^{\infty}(\D,dA)$.

{\bf Compactness.} By above, $GV_{(\phi,g)}: A^p_\omega\rightarrow A^q_\omega$ is compact if and only if the measure $\nu_{\phi,\omega,g}$ is a vanishing $q$-Carleson measure for $A_\omega^p$. This is equivalent to 
$$
	\lim_{|z|\to1^-}\tau(z)^{2(1-q/p)}\int_{\D} |k_{q,z}(\xi)|^q\omega(\xi)^{q/2}\,d\nu_{\phi,\omega,g}(\xi)=0.
$$
Now, using (a) of Lemma \ref{lem:RK-PE1}, we get 
\begin{align*}
	\tau(z)^{2(1-q/p)}\int_{\mathbb{D}} |k_{q,z}(\xi)|^q\, \omega(\xi)^{q/2} \,d\nu_{\phi,\omega,q}(\xi)
	&\asymp \int_{\mathbb{D}} |k_{p,z}(\xi)|^q\, \omega(\xi)^{q/2} \,d\nu_{\phi,\omega,q}(\xi)\\
	&=   \int_{\mathbb{D}} |k_{p,z}(\phi(\xi))|^q \,  \frac{|g(z)|^q}{(1+{\varphi}^{'}(\xi))^q}\,\omega(z)^{q/2}\,dA(\xi)\\&=GB^{\phi}_{0,p,q}.
\end{align*}
Therefore, $\lim_{|z|\to1^-}GB^\phi_{0,p,q}(g)=0$ if and only if $GV_{(\phi,g)}$ is compact.

\subsection{Proof of Theorem~\ref{thmV1} (B)} {\bf Boundedness.} Let $0<p<q=\infty$ and suppose that \eqref{eq-V2} holds. Then, by \eqref{littleW1}, we have
\begin{align*}
	\|GV_{(\phi,g)}f\|_{A^\infty_{\om}}&\asymp \sup_{z\in \D} \frac{|f(\phi(z))||g(z)|}{(1+\varphi'(z))}\,\,\omega(z)^{\frac{1}{2}}\\
	&\le \sup_{z\in \D} MV_{g,\phi,\omega}(z)\,\sup_{z\in \D}{|f(\phi(z))|\omega(\phi(z))^{\frac{1}{2}}}\,\Delta\varphi(\phi(z))^{-1/p}\\
	&\lesssim \sup_{z\in \D} MV_{g,\phi,\omega}(z)\,\sup_{z\in \D}{|f(\phi(z))|\omega(\phi(z))^{\frac{1}{2}}}\,\tau(\phi(z))^{2/p}.
\end{align*}
By \eqref{Eq-gamma} for $f\in A_{\omega}^p$, we obtain
\begin{align*}
	\|GV_{(\phi,g)}f\|_{A^\infty_{\om}}&\lesssim \sup_{z\in \D}\left(\int_{D_{\delta}(\phi(z))}{|f(\xi)|^p\omega(\xi)^{\frac{p}{2}}}\, dA(\xi)\right)^{1/p}\\
	&\le\left(\int_{\D}{|f(\xi)|^p\omega(\xi)^{\frac{p}{2}}}\, dA(\xi)\right)^{1/p}
	= \|f\|_{A^p_{\om}}.
\end{align*}
Therefore, the operator $GV_{(\phi, g)}$ is bounded.

Conversely, suppose that the  operator $GV_{ (\phi, g)}: A^p_\omega\rightarrow A^\infty_\omega$ is bounded. Taking $\xi\in\D$ such that $|\phi(\xi)|>\rho_0,$ we consider the function $f_{\phi(\xi),n,p}$ given by  $f_{\phi(\xi),n,p}:=\frac{F_{\phi(\xi),n,p}}{\tau(\phi(\xi))^{2/p}}$  where $ F_{\phi(\xi),n,p}$ is the test function defined in Lemma \ref{Borichevlemma}. These  functions $f_{\phi(\xi),n,p}$ belong to $A^p_{\om}$ with $\|f_{\phi(\xi),n,p}\|_{A^p_{\om}}\asymp 1.$ By {\eqref{littleW1}},
\begin{align*}
	\infty >\|GV_{(\phi,g)}(f_{\phi(\xi),n,p})\|_{A^\infty_{\om}}
	&\geq  \sup_{z\in \D}\frac{|f_{\phi(\xi),n,p}(z)||g(z)|}{(1+\varphi'(z))}\,\omega(z)^{\frac{1}{2}}\\
	&\geq\sup_{z\in \D}\frac{|F_{\phi(\xi),n,p}(z)||g(z)|}{\tau(\phi(\xi))^{2/p}(1+\varphi'(z))}\,\omega(z)^{\frac{1}{2}}\\
	&\geq\frac{|F_{\phi(\xi),n,p}(\phi(\xi))||g(\xi)|}{\tau(\phi(\xi))^{2/p}(1+\varphi'(\xi))}\,\omega(\xi)^{\frac{1}{2}}\,\frac{\omega(\phi(\xi))^{\frac{1}{2}}}{\omega(\phi(\xi))^{\frac{1}{2}}}.
\end{align*} 
In this case, by \eqref{BL1},
\begin{equation}\label{qM}
\begin{split}
	\infty >\|GV_{(\phi,g)}(f_{\phi(\xi),n,p})\|_{A^\infty_{\om}}
	&\geq\frac{|g(\xi)|}{(1+\varphi'(\xi))}\,\frac{\omega(\xi)^{\frac{1}{2}}}{\omega(\phi(\xi))^{\frac{1}{2}}}\,\tau(\phi(\xi))^{-2/p}\\
	&\asymp \frac{|g(\xi)|}{(1+\varphi'(\xi))}\,\frac{\omega(\xi)^{\frac{1}{2}}}{\omega(\phi(\xi))^{\frac{1}{2}}}\,\Delta\varphi(\phi(\xi))^{1/p}
	=MV_{g,\phi,\omega}(\xi).
\end{split}
\end{equation}
On the other hand, if we define $f(z)=z$ and use the boundedness of the operator $GV_{(\phi,g)}: A^p_\omega\rightarrow A^\infty_\omega,$ we obtain
\begin{equation}\label{333333}
	\|GV_{(\phi,g)}f\|_{A^\infty_{\om}} 
	\asymp\sup_{z\in \D}\frac{|g(z)|}{(1+\varphi'(z))}\,\,\omega(z)^{\frac{1}{2}}
	\lesssim \|f\|_{A^p_{\om}} <\infty.
\end{equation}
Therefore, in the case of $|\phi(\xi)|\le\rho_0, \, \xi \in \D,$ we have 
\begin{align*}
	\frac{|g(\xi)|}{(1+\varphi'(\xi))}\,\frac{\omega(\xi)^{\frac{1}{2}}}{\omega(\phi(\xi))^{\frac{1}{2}}}\,\Delta\varphi(\phi(\xi))^{1/p}
	&\asymp \frac{|g(\xi)|}{(1+\varphi'(\xi))}\,\frac{\omega(\xi)^{\frac{1}{2}}}{\omega(\phi(\xi))^{\frac{1}{2}}}\,\tau(\phi(\xi))^{-2/p}\\
	&\le C_1 \,\frac{|g(\xi)|}{(1+\varphi'(\xi))}\,\,\omega(\xi)^{\frac{1}{2}}<\infty,
\end{align*}
where\
$$
	C_1 = \sup_{|\phi(\xi)|\le \rho_0}\Big\{\omega(\phi(\xi))^{\frac{-1}{2}} \,\tau(\phi(\xi))^{-2/p}\Big\}<\infty.
$$
This, \eqref{eq-V2} holds.

{\bf Compactness.} Suppose that the operator $GV_{(\phi,g)}: A^p_\omega\rightarrow A^\infty_\omega$ is compact and {define}
$$
	f_{\phi(\xi),n,p}:=\frac{F_{\phi(\xi),n,p}}{\tau(\phi(\xi))^{2/p}}, \quad \textrm{for}\,\,\, |\phi(\xi)|>\rho_0,
$$
which are in $A^p_\omega$ and converge uniformly to zero on compact subsets of $\D$ as $|\phi(\xi)|\to 1$. Thus, 
$$
	\|GI_{(\phi,g)}(f_{\phi(\xi),n,p})\|_{A^\infty_{\om}}\rightarrow 0
$$
as $|\phi(\xi)|\to 1.$ Thus, \eqref{qM} shows that
\begin{equation*}
	\lim_{|\phi(\xi)|\to1^-} MV_{g,\phi,\om}(\xi)
	\lesssim\lim_{|\phi(\xi)|\to1^-}\|GV_{(\phi,g)}(f_{\phi(\xi),n,p})\|_{A^\infty_{\om}} = 0.
\end{equation*}

Conversely, if $\{f_n\}$ is a bounded sequence of functions in $A^p_\omega$ converging uniformly to zero on compact subsets of $\D$, then, as for $GI_{(\phi,g)}$, it follows that 
$$
	\|GV_{(\phi,g)}(f_n)\|_{A^\infty_{\om}}=\sup_{\xi\in \D}\frac{|f_n(\phi(\xi))||g(\xi)|}{1+\varphi'(\xi)}\,\,\omega(\xi)^{\frac{1}{2}}\rightarrow 0,\quad n\to \infty,
$$
which proves the compactness of the operator $GV_{(\phi,g)}: A^p_\omega\rightarrow A^\infty_\omega$.

\subsection{Proof of Theorem~\ref{thmV1} (C)} {\bf Boundedness.} Let $p=q=\infty$. Suppose first that $\eqref{eq-G3-NV}$ holds. Then, by \eqref{littleW1}, 
\begin{align*}
	\|GV_{(\phi,g)}f\|_{A^\infty_{\om}}&=\sup_{z\in \D} \frac{|f(\phi(z))||g(z)|}{(1+\varphi'(z))}\,\,\omega(z)^{\frac{1}{2}}\\
	&\le \sup_{z\in \D} NV_{g,\phi,\omega}(z)\,\sup_{z\in \D}{|f(\phi(z))|\omega(\phi(z))^{\frac{1}{2}}}\\
	&\le \sup_{z\in \D} NV_{g,\phi,\omega}(z)\,\sup_{z\in \D}{|f(z)|\omega(z)^{\frac{1}{2}}}\lesssim  \|f\|_{A^\infty_\omega},
\end{align*}
that is, $GV_{(\phi, g)}$ bounded.

Conversely, suppose that $GV_{(\phi,g)}: A^\infty_\omega\rightarrow A^\infty_\omega$ is bounded and  {show} that $\eqref{eq-G3-NV}$  holds. As before, if $\xi\in\D$ is such that $|\phi(\xi)|>\rho_0$, we use the test  functions $F_{\phi(\xi),n,p}$ to obtain
\begin{align*}
	\infty >\|GV_{(\phi,g)}(F_{\phi(\xi),n,p})\|_{A^\infty_{\om}}
	&=\sup_{z\in \D}\frac{|F_{\phi(\xi),n,p}(z)||g(z)|}{(1+\varphi'(z))}\,\omega(z)^{\frac{1}{2}}\\
	&\geq\frac{|F_{\phi(\xi),n,p}(\phi(\xi))||g(\xi)|}{(1+\varphi'(\xi))}\,\omega(\xi)^{\frac{1}{2}}\, \frac{\omega(\phi(\xi))^{\frac{1}{2}}}{\omega(\phi(\xi))^{\frac{1}{2}}}\\
	&\geq NV_{g,\phi,\omega}(\xi)\, |F_{\phi(\xi),n,p}(\phi(\xi))|\,\omega(\phi(\xi))^{\frac{1}{2}}.
\end{align*}
Now
\begin{equation}\label{eqM11}
\begin{split}
	\infty >\|GV_{(\phi,g)}(F_{\phi(\xi),n,p})\|_{A^\infty_{\om}}
	&\geq\frac{ |g(\xi)|}{(1+\varphi'(\xi))}\,\frac{\omega(\xi)^{\frac{1}{2}}}{\omega(\phi(\xi))^{\frac{1}{2}}}\\
	&\asymp \frac{ |g(\xi)|}{(1+\varphi'(\xi))}\,\frac{\omega(\xi)^{\frac{1}{2}}}{\omega(\phi(\xi))^{\frac{1}{2}}}=NV_{g,\phi,\omega}(\xi).
\end{split}
\end{equation}

If $f(z)=z$, the boundedness of the operator $GV_{(\phi,g)}: A^\infty_\omega\rightarrow A^\infty_\omega$ implies that
\begin{equation}\label{q-xt}
	\|GV_{(\phi,g)}f\|_{A^\infty_{\om}}
	=\sup_{z\in \D}\frac{|g(z)|}{(1+\varphi'(z))}\,\,\omega(z)^{\frac{1}{2}}\lesssim \|f\|_{A^\infty_{\om}} <\infty.
\end{equation}
Therefore, in the case of $|\phi(\xi)|\le\rho_0, \, \xi \in \D,$ we have 
\begin{align*}
	\frac{ |g(\xi)|}{(1+\varphi'(\xi))}\,\frac{\omega(\xi)^{\frac{1}{2}}}{\omega(\phi(\xi))^{\frac{1}{2}}}
	&\le C_2 \,\frac{|g(\xi)|}{(1+\varphi'(\xi))}\,\,\omega(\xi)^{\frac{1}{2}}<\infty,
\end{align*}
where
$$
	C_2 = \sup_{|\phi(\xi)|\le \rho_0}\Big\{\omega(\phi(\xi))^{\frac{-1}{2}} \Big\}<\infty.
$$
Combining this with \eqref{eqM11} shows that \eqref{eq-G3-NV} holds. 

{\bf Compactness.} This is similar to the proof of (C) of Theorem~\ref{thm1}.

\subsection{Proof of Theorem~\ref{thmV1} (D)} Let $0<q<p<\infty$ and suppose that $GV_{(\phi,g)}: A^p_\omega\rightarrow A^q_\omega$ is bounded. According to \eqref{eq-eqv}, the measure  $\nu_{\phi,\omega,g}$ is a $q$-Carleson measure for $A_\omega^p$. Thus, by Theorem \ref{thm:VCMQP},  $\nu_{\phi,\omega,g}$ is a vanishing $q$-Carleson measure for $A_\omega^p$. In this case, we have
$$
	\|GV_{(\phi,g)}f_n\|^q_{A^q_{\om}}\rightarrow 0, \quad n\to \infty,
$$
for any sequence $\{f_n\}\subset A^p_\omega$ converges  to zero uniformly on compact subsets of $\D.$ By Lemma $3.7$ of \cite{tija}, $GV_{(\phi,g)}$ is compact.

Next we show that (a) and (c) are equivalent. Suppose first that (c) holds. Then
\begin{equation}\label{equ111}
\begin{split}
	\int_{\D} G_{q}(v_{\phi,\omega,q})(z)^{p/(p-q)} dA(z)
	&=\int_{\D} \Big(\tau(z)^{2(1-\frac{q}{p})}G_{q}(v_{\phi,\omega,q})(z)\Big)^{p/(p-q)}d\lambda(z)\\ &\asymp \int_{\D}GB^\phi_{0,p,q}(g)^{p/(p-q)}d\lambda(z)\, {<\infty}.
\end{split}
\end{equation}
Thus, according to Theorem \ref{thm:CMOP}, $\nu_{\phi,\omega,q}$ is a $q$-Carleson measure for $A_{\omega}^p$. Then, by \eqref{littleW},
$$
	\|GV_{(\phi,g)}f_n\|^q_{A^q_{\om}}
	\asymp \int_{\D} |f(z)|^q\, {\omega(z)^{q/2}}\,d\nu_{\phi,\omega,g}(z)
	\lesssim \|f\|^q_{A_{\omega}^p},
$$
for any function $f\in A_{\omega}^p$.

Conversely, suppose the operator $GV_{ (\phi, g)}: A^p_\omega\rightarrow A^q_\omega$ is bounded. Then, for each function  $f\in A_{\omega}^p$, by \eqref{littleW},
$$
	\|GV_{ (\phi, g)}\,f\|_{A^q_{\omega}}^q
	\asymp \int_{\D} |f(z)|^q\, {\omega(z)^{q/2}}\,d\nu_{\phi,\omega,g}(z).
$$
Thus, the measure $\nu_{\phi,\omega,g}$ is a $q$-Carleson measure for $A_\omega^p$. According to Theorem \ref{thm:CMOP}, $\nu_{\phi,\omega,g}$ belongs to $L^{p/(p-q)}(\D,dA)$. Combining this with \eqref{equ111} yields that $GB^\phi_{0,p,q}(g)\in L^{p/(p-q)}(\D,d\lambda)$.

Let $0<q<p=\infty$ and suppose that $GV_{(\phi,g)}: A^\infty_\omega\rightarrow A^q_\omega$ is bounded. Then, by \eqref{littleW},
\begin{equation*}
	\|GV_{(\phi,g)}f\|^q_{A^q_{\om}}
	\asymp\int_{\D}\frac{|f(\phi(z))|^q|g(z)|^q}{(1+\varphi'(z))^q}\,\,\omega(z)^{\frac{q}{2}}\, dA(z)
	\lesssim \|f\|^q_{A_\omega^\infty},
\end{equation*}
and it follows from Theorem \ref{thm:CMOP2} that the measure $\nu_{\phi,\omega,g}$ is a $q$-Carleson measure for $A_\omega^\infty$. Thus, by Theorem \ref{thm:VCMQP},  $\nu_{\phi,\omega,g}$ is a vanishing $q$-Carleson measure for $A_\omega^\infty$. As in the previous case, this shows the compactness of the operator $GV_{(\phi,g)}$. 

It remains to prove that $(1)$ and $(3)$ are equivalent {when $p=\infty$}. Assume first that  $(3)$  holds. Then
\begin{equation}\label{eq111}
\begin{split}
	\int_{\D} G_{q}(\nu_{\phi,\omega,q})(z) \,dA(z)
	&=\int_{\D} \Big(\tau(z)^{2}G_{q}(\nu_{\phi,\omega,q})(z)\Big)\,d\lambda(z)\\ 
	&\asymp \int_{\D}GB^\phi_{0,p,q}(g)(z)\,d\lambda(z).
\end{split}
\end{equation}
Thus, according to Theorem \ref{thm:CMOP2}, $\nu_{\phi,\omega,q}$ is a $q$-Carleson measure for $A_{\omega}^\infty$. Then for any function $f\in A_{\omega}^\infty$, we have
$$
	\|GV_{(\phi,g)}f_n\|^q_{A^q_{\om}}
	\asymp \int_{\D} |f(z)|^q\, {\omega(z)^{q/2}}\,d\nu_{\phi,\omega,g}(z)
	\lesssim \|f\|^q_{A_{\omega}^\infty}.
$$

Conversely, suppose the operator $GV_{ (\phi, g)}: A^\infty_\omega\rightarrow A^q_\omega$ is bounded. Then, for any function  $f\in A_{\omega}^\infty$, we have
$$
	\|GV_{ (\phi, g)}\,f\|_{A_{\omega}}^q\asymp\int_{\D} |f(z)|^q\, {\omega(z)^{q/2}}\,d\nu_{\phi,\omega,g}(z).
$$
By assumption, this implies that the measure $\nu_{\phi,\omega,g}$ is a $q$-Carleson measure for $A_\omega^\infty$. According to Theorem \ref{thm:CMOP2}, $\nu_{\phi,\omega,g}$ belongs to $L^{1}(\D,dA)$. Combining this with \eqref{eq111} implies that $GB^\phi_{0,p,q}(g)\in L^{1}(\D,d\lambda)$.

\section{Proofs of Proposition \ref{prop1} and Corollary \ref{cor1}}\label{proofs2}

\subsection{Proof of Proposition \ref{prop1} (A)} Suppose that the  operator $GI_{(\phi,g)}: A^p_\omega\rightarrow A^q_\omega$ is bounded. Let $\xi\in\D$ be such that $|\phi(\xi)|>\rho_0$. Using the test function of Lemma \ref{Borichevlemma}, \eqref{Eq-gamma} and \eqref{littleW}, we get 
\begin{align*}
	\|F_{\phi(\xi),n,p}\|^q_{A^p_{\om}} &\gtrsim\|GI_{(\phi,g)}F_{\phi(\xi),n,p}\|^q_{A^q_{\om}}\asymp
\int_{\D}\frac{|F'_{\phi(\xi),n,p}(\phi(z))|^q}{(1+\varphi'(z))^q}\,|g(z)|^q\,\omega(z)^{\frac{q}{2}}\, dA(z)\\
	&\gtrsim \tau(\xi)^2\frac{|F'_{\phi(\xi),n,p}(\phi(\xi))|^q}{(1+\varphi'(\xi))^q}\,|g(\xi)|^q\,\omega(\xi)^{\frac{q}{2}},
\end{align*}
while Lemma \ref{BL3} implies that
\begin{equation*}
	\|F_{\phi(\xi),n,p}\|^q_{A^p_{\om}} 
	\gtrsim \tau(\xi)^2|g(\xi)|^q\,\frac{(1+\varphi'(\phi(\xi)))^q}{(1+\varphi'(\xi))^q}\,\frac{\omega(\xi)^{\frac{q}{2}}}{\omega(\phi(\xi)^{\frac{q}{2}}}.
\end{equation*}
By Lemma \ref{Borichevlemma}, we have 
\begin{equation}\label{eqp1}
	1\gtrsim|g(\xi)|\, \frac{\tau(\xi)^{2/q}}{\tau(\phi(\xi))^{2/p}}\,\frac{1+\varphi'(\phi(\xi))}{1+\varphi'(\xi)}\,\frac{\omega(\xi)^{\frac{1}{2}}}{\omega(\phi(\xi))^{\frac{1}{2}}}.
\end{equation}
When $|\phi(\xi)|\leq \rho_0$, we have 
$$
	\sup_{\phi(\xi)\le\rho_0 }|g(\xi)|\, \frac{\tau(\xi)^{2/q}}{\tau(\phi(\xi))^{2/p}}\,\frac{1+\varphi'(\phi(\xi))}{1+\varphi'(\xi)}\,\frac{\omega(\xi)^{\frac{1}{2}}}{\omega(\phi(\xi))^{\frac{1}{2}}}<\infty.
$$
Thus, \eqref{eq-nc1-GI} holds.

Suppose next that the  operator $GI_{(\phi,g)}: A^p_\omega\rightarrow A^q_\omega$ is compact. Let $\xi\in\D$ be such that $|\phi(\xi)|>\rho_0$ and  define
$$
	f_{\phi(\xi),n,p}
	=\frac{F_{\phi(\xi),n,p}}{\tau(\phi(\xi))^{2/p}},\quad \textrm{for}\,\,\, |\phi(\xi)|>\rho_0,
$$
which belongs to $A^p_\omega$ and converges uniformly to zero on compact subsets of $\D$ as $|\phi(\xi)|\to 1.$ By \eqref{Eq-gamma} and Lemma \ref{BL3}, we get 
\begin{align*}
	\|GI_{(\phi,g)}f_{\phi(\xi),n,p}\|^q_{A^q_{\om}}
	&\asymp \int_{\D}\frac{|f'_{\phi(\xi),n,p}(\phi(z))|^q}{(1+\varphi'(z))^q}\,|g(z)|^q\,\omega(z)^{\frac{q}{2}}\, dA(z)\\&\gtrsim \tau(\xi)^2\frac{|f'_{\phi(\xi),n,p}(\phi(\xi))|^q}{(1+\varphi'(\xi))^q}\,|g(\xi)|^q\,\omega(\xi)^{\frac{q}{2}}\\
	&\gtrsim|g(\xi)|^q\,\frac{\tau(\xi)^2}{\tau(\phi(\xi))^{2q/p}}\frac{(1+\varphi'(\phi(\xi)))^q}{(1+\varphi'(\xi))^q}\,\frac{\omega(\xi)^{\frac{q}{2}}}{\omega(\phi(\xi))^{\frac{q}{2}}}.
\end{align*}
Using the  compactness of the operator  $GI_{(\phi,g)} ,$ we have the desired conclusion and  the proof is complete.

\subsection{Proof of Proposition \ref{prop1} (B)} 

Suppose that $GV_{(\phi,g)}: A^p_\omega\rightarrow A^q_\omega$ is bounded.  By Theorem~\ref{thmV1} (A), this is equivalent  to $GB^\phi_{0,p,q}(g)\in L^{\infty}(\D,dA).$ By  \eqref{Eq-gamma} and \eqref{eqn:RK-Diag1}, we have 
\begin{equation}\label{EsBV}
\begin{split}
	GB^\phi_{0,p,q}(g)(\phi(z))&=\int_{\D} |k_{p,\phi(z)}(\phi(\xi))|^q\,\frac{|g(\xi)|^q}{(1+\varphi'(\xi))^{q}}\,\,\omega(\xi)^{q/2}\, dA(\xi)\\
	&\geq \int_{D_{\delta}(z)} |k_{p,\phi(z)}(\phi(\xi))|^q\,\frac{|g(\xi)|^q}{(1+\varphi'(\xi))^{q}}\,\,\omega(\xi)^{q/2}\, dA(\xi)\\
	&\gtrsim \tau(z)^{2} \, |k_{p,\phi(z)}(\phi(z))|^q\,\frac{|g(z)|^q}{(1+\varphi'(z))^q}\,\,\omega(z)^{q/2}\\
	&\gtrsim\frac{ \tau(z)^{2}}{\tau(\phi(z))^{2q/p}} \,\frac{\omega(z)^{q/2}}{\omega(\phi(z))^{q/2}} \,\frac{|g(z)|^q}{(1+\varphi'(z))^q},
\end{split}
\end{equation}
which proves that \eqref{eq-nc1-GV} holds. If $GV_{(\phi, g)}$ is compact, then it follows from Theorem~\ref{thmV1} (A) that $GB^\phi_{0,p,q}(g)(\phi(z))\to 0$ as $|z|\to 1$, which completes the proof.

\subsection{Proof of Corollary \ref{cor1}} (A) Let $p<q$ and suppose that $GI_{(id,g)}$ is bounded. By  \eqref{eqn:RK-Diag1} and Lemma \ref{lem:subHarmP}, we have
\begin{align*}
	|g(z)|^q &\asymp \tau(z)^{2q/p}\,  |g(z)|^q\,|k_{p,z}(z)|^q\,\omega(z)^{\frac{q}{2}} \\
	&\lesssim \frac{\tau(z)^{2q/p}}{\tau(z)^2}\int_{D_{\delta}(z)} |g(s)|^q\,|k_{p,z}(s)|^q\,\omega(s)^{\frac{q}{2}} dA(s)\lesssim  \frac{\tau(z)^{2q/p}}{\tau(z)^2}\, GB^{id}_{1,p,q}(g)(z).
\end{align*}
Then, using the boundedness of $GI_{(id,g)},$ we obtain 
\begin{equation*}
	\sup_{z\in \D}|g(z)|^q \,\tau(z)^{2(1-q/p)}
	\lesssim \sup_{z\in \D} GB^{id}_{1,p,q}(g)(z)< \infty.
\end{equation*}
Since $\tau(z)^{2(1-q/p)}\rightarrow \infty,$ as $|z|\rightarrow1$, the function $g$ must be zero.

\subsection{Proof of Corollary \ref{cor1}} (B) Let $q<p$. {Using} 
\begin{equation}\label{e-eqv}
	\|GI_{(\phi,g)}f\|^q_{A^q_{\om}}
	\asymp\int_{\D}\frac{|f'(\phi(z))|^q|g(z)|^q}{(1+\varphi'(z))^q}\,\,\omega(z)^{\frac{q}{2}}\, dA(z)
	=\|f'\|^q_{L^q(\mu_{\phi,\omega,g})}
\end{equation}
(see \eqref{littleW}) and Lemma \ref{lem2}, we get
$GI_{(\phi,g)}: A^p_\omega\rightarrow A^q_\omega$ is bounded if and only if $I_{\mu_{\phi,\omega,g}}: S^p_\omega\rightarrow L^q(\mu_{\phi,\omega,g})$ is bounded if and only if $I_{\mu_{\phi,\omega,g}}: S^p_\omega\rightarrow L^q(\mu_{\phi,\omega,g})$ is compact if and only if the function 
\begin{equation}\label{FC}
	F_{\delta,\mu_{\phi,\omega,g}}(\varphi)(z):= \frac{1}{\tau(z)^{2}}\int_{D_{\delta}(z)} (1+\varphi'(\xi))^q\omega(\xi)^{-q/2}\,d\mu_{\phi,\omega,g}(\xi) 
\end{equation}
belongs to $ L^{p/(p-q)}(\D,dA).$ Since $\phi=id,$  we have
$$
	d\mu_{\phi,\omega,g}(z)=\frac {|g(z)|^q}{(1+\varphi'(z))^q}\,\,\omega(z)^{q/2}\, dA(z)
$$
and invoking this in the condition \eqref{FC}, it  becomes exactly
$$
	\frac{1}{\tau(z)^{2}}\int_{D_{\delta}(z)}|g(\xi)|^q \,dA(\xi)
	\in L^{p/(p-q)}(\D,dA).
$$
Applying Lemma \ref{lem:subHarmP}, we get that $g\in L^{r}(\D,dA),$ with $r=pq/(p-q).$

Conversely, suppose that $g\in L^{r}(\D,dA)$.  By H\"older's inequality and \eqref{littleW}, we obtain
\begin{equation}\label{eq-eqvn}
\begin{split}
	\|GI_{(id,g)}f\|^q_{A^q_{\om}}&\asymp\int_{\D}\frac{|f'(z)|^q|g(z)|^q}{(1+\varphi'(z))^q}\,\,\omega(z)^{\frac{q}{2}}\, dA(z)\\
	&\lesssim\left(\int_{\D}\frac{|f'(z)|^p\omega(z)^{\frac{p}{2}}}{(1+\varphi'(z))^p}\, dA(z)\right)^{\frac{q}{p}}\left(\int_{\D}|g(z)|^r\, dA(z)\right)^{\frac{q}{r}}\\
	&\asymp\|f\|^q_{A^p_{\om}}\|g\|^q_{L^r(\D, dA)}\lesssim\|f\|^q_{A^p_{\om}},
\end{split}
\end{equation}
which proves boundedness and completes the proof.

\subsection{Proof of Corollary \ref{cor1} (C)}  Let  $ 0<p\le q\le \infty$. We characterize boundedness using Theorem \ref{thmV1}. Suppose that $GB^{id}_{0,p,q}(g')\in L^{\infty}(\D,dA)$. It follows from \eqref{EsBV} (changing $g$ by $g'$ and $\phi=id$),
\begin{equation}\label{EsBVn}
\begin{split}
GB^{id}_{0,p,q}(g')(z)&=\int_{\D} |k_{p,z}(\xi)|^q\,\frac{|g'(\xi)|^q}{(1+\varphi'(\xi))^{q}}\,\,\omega(\xi)^{q/2}\, dA(\xi)\\
&\gtrsim\frac{ \tau(z)^{2}}{\tau(z)^{2q/p}} \,\frac{|g'(z)|^q}{(1+\varphi'(z))^q}\asymp\left(\frac{|g'(z)|}{(1+\varphi'(z))}\, \Delta\varphi(z)^{\frac{1}{p}-\frac{1}{q}}\right)^{q}.
\end{split}
\end{equation}
Thus, $$\frac{|g'(z)|}{(1+\varphi'(z))}\Delta\varphi(z)^{\frac{1}{p}-\frac{1}{q}}\in L^{\infty}(\D,dA).$$

Conversely, suppose that $$T(g,\varphi)(z):=\frac{|g'(z)|}{(1+\varphi'(z))}\, \Delta\varphi(z)^{\frac{1}{p}-\frac{1}{q}}\in L^{\infty}(\D,dA).$$ By \eqref{eqn:Eq-NE1}, we have 
\begin{align*}
	GB^{id}_{0,p,q}(g')(z)&=\int_{\D} |k_{p,z}(\xi)|^q\,\frac{|g'(\xi)|^q}{(1+\varphi'(\xi))^{q}}\,\,\omega(\xi)^{q/2}\, dA(\xi)\\
	&\lesssim \left( \tau(z)^{2(1-q/p)} \int_{\D} |k_{q,z}(\xi)|^q\,\Delta\varphi(z)^{1-\frac{q}{p}}\,\omega(\xi)^{q/2}\, dA(\xi)\right)\sup_{z\in \D}\,(T(g,\varphi)(z))^q.
\end{align*}
Since $\Delta\varphi(z)\asymp \tau(z)^{-2}$,
\begin{equation}\label{EsBV2}
\begin{split}
	GB^{id}_{0,p,q}(g')(z)
	&\lesssim \left( \int_{\D} |k_{q,z}(\xi)|^q\,\omega(\xi)^{q/2}\, dA(\xi)\right)\sup_{z\in \D}\,(T(g,\varphi)(z))^q\\
	&=\|k_{q,z}\|^q_{A^q_{\om}}\sup_{z\in \D}\,(T(g,\varphi)(z))^q=\sup_{z\in \D}\,(T(g,\varphi)(z))^q.
\end{split}
\end{equation}
This finishes the proof of boundedness.

The characterization for compactness follows from Theorem \ref{thmV1}, \eqref{EsBV2}, and \eqref{EsBVn}.

\subsection{Proof of Corollary \ref{cor1}} (D) Let $0<q<p<\infty$. We first suppose that $V_g: A^p_\omega\to A^q_\omega$ is bounded, that is, $GB^{id}_{0,p,q}(g') \in L^{\frac{p}{p-q}}(\D,d\lambda)$ (see Theorem~\ref{thmV1}). Then, by \eqref{Eq-gamma}, we have
\begin{align*}
	GB^{id}_{0,p,q}(g')(z)
	&=\int_{\D} |k_{p,z}(\xi)|^q\,\frac{|g'(\xi)|^q}{(1+\varphi'(\xi))^{q}}\,\,\omega(\xi)^{q/2}\, dA(\xi)\\
	&\gtrsim \tau(z)^{2} \, |k_{p,z}(z)|^q\,\frac{|g'(z)|^q}{(1+\varphi'(z))^{q}}\,\omega(z)^{q/2}.
\end{align*}
By  Lemma \ref{lem:RK-PE1}, we obtain
\begin{equation*}
	GB^{id}_{0,p,q}(g')(z)
	\gtrsim\frac{ \tau(z)^{2}}{\tau(z)^{2q/p}} \,\frac{|g'(z)|^q}{(1+\varphi'(z))^{q}}
	\asymp\left(\frac{|g'(z)|}{(1+\varphi'(z))}\, \Delta\varphi(z)^{\frac{1}{p}-\frac{1}{q}}\right)^{q}.
\end{equation*}
In this case, we extract that 
\begin{equation*}
	\tau(z)^{2(\frac{q}{p}-1)}GB^{id}_{0,p,q}(g')(z)\gtrsim\left(\frac{|g'(z)|}{(1+\varphi'(z))}\right)^{q}.
\end{equation*}
By our assumption and the fact that $\tau(z)^{2q/p}$ is bounded, it follows that \eqref{eq-V9} holds.

Conversely, put $r=\frac{pq}{p-q}.$ By H\"older's inequality, we obtain
\begin{align*}
	&GB^{id}_{0,p,q}(g')(z)^{p/(p-q)}\\
	&=\left(\int_{\D} |k_{p,z}(\xi)|^q\,\frac{|g'(\xi)|^q}{(1+\varphi'(\xi))^{q}}\,\,\omega(\xi)^{q/2}\, dA(\xi)\right)^{p/(p-q)}\\
	&\le\|K_z\|^{-r}_{A^p_{\om}} \left(\int_{\D} |K_{z}(\xi)|^{\frac{r}{2}} \left(\frac{|g'(\xi)|}{1+\varphi'(\xi)}\right)^{r} \omega(\xi)^{\frac{r}{4}} \, dA(\xi)\right)
\cdot\left(\int_{\D} |K_{z}(\xi)|^{\frac{p}{2}}  \omega(\xi)^{\frac{p}{4}}\, dA(\xi)\right)^{\frac{q}{(p-q)}}\\
	&=\frac{\|K_z\|^{r/2}_{A^{p/2}_{\om}}}{\|K_z\|^{r}_{A^p_{\om}}} \int_{\D} |K_{z}(\xi)|^{\frac{r}{2}} \left(\frac{|g'(\xi)|}{1+\varphi'(\xi)}\right)^{r} \omega(\xi)^{\frac{r}{4}} \, dA(\xi).
\end{align*}
By Theorem \ref{RK-PE}, $\frac{\|K_z\|^{r/2}_{A^{p/2}_{\om}}}{\|K_z\|^{r}_{A^p_{\om}}}\asymp \omega(z)^{\frac{r}{4}}\,\tau(z)^{r}$, and Fubini's theorem implies that
\begin{multline*}\label{EsBV4}
	\int_{\D} GB^{id}_{0,p,q}(g')(z)^{p/(p-q)}\frac{dA(z)}{\tau(z)^{2}}\\
	\lesssim\int_{\D}\left(\frac{|g'(\xi)|}{1+\varphi'(\xi)}\right)^{r} \omega(\xi)^{\frac{r}{4}}  \left( \int_{\D} |K_{\xi}(z)|^{\frac{r}{2}} \, \omega(z)^{\frac{r}{4}}\,\tau(z)^{r-2} \,dA(z) \right)dA(\xi).
\end{multline*}
Since
$$
	\omega(\xi)^{\frac{r}{4}}  \left( \int_{\D} |K_{\xi}(z)|^{\frac{r}{2}} \, \omega(z)^{\frac{r}{4}}\,\tau(z)^{r-2} \,dA(z) \right)\lesssim 1
$$
(see Lemma \ref{nEstim}), the proof is complete.

\subsection{Proof of Corollary \ref{cor2}} (I) Let $0<p=q<\infty$. By (c) of Lemma 32 in \cite{CoPe},
\begin{equation}\label{eq-V10}
	\psi_\om(r) \asymp (1+\varphi'(r))^{-1} \quad \textrm{for} \,\, r \in [0,1).
\end{equation}
Therefore,
\begin{equation}\label{CJ}
\begin{split}
	GB^{id}_{0,p,p}(g')(z)&=\int_{\D} |k_{p,z}(\xi)|^p\,\frac{|g'(\xi)|^p}{(1+\varphi'(\xi))^{p}}\,\,\omega(\xi)^{p/2}\, dA(\xi)\\
	&\asymp \sup_{\xi\in \D}\left(\psi_\om(\xi)|g'(\xi)|\right)^p \left(  \int_{\D} |k_{p,z}(\xi)|^p\,\omega(\xi)^{p/2}\, dA(\xi)\right)\\
	&=\sup_{\xi\in \D}\left(\psi_\om(\xi)|g'(\xi)|\right)^p \|k_{p,z}\|^p_{A^p_\om}=\sup_{\xi\in \D}\left(\psi_\om(\xi)|g'(\xi)|\right)^p.
\end{split}
\end{equation}
The other assertion follows easily from \eqref{CJ}. 

\subsection{Proof of Corollary \ref{cor2}} (II) Let $0<p<q<\infty$. Note that the  weighted Bergman space $A^p(\omega),$ defined in \cite{PP1}, is  the same as  the  Bergman spaces $A^p_W, $ with $W=\omega^{2/p}$. Moreover, 
$$
	GB^{id}_{0,p,q}(g')(z) 
	= \int_{D_{\delta}(z)}|k_{p,z}(\xi)|^q\frac{|g'(\xi)|^q}{(1+\varphi'(\xi))^{q}}\,\omega(\xi) \,dA(\xi),
$$
and \eqref{eqn:RK-Diag1} is transformed to 
\begin{equation}\label{eqn:RK-Diag1n}
	|k_ {p,z}(\zeta)|^q \, \omega(\zeta)^{q/p} \asymp \tau(z)^{-2q/p} ,\qquad \zeta \in D_\delta(z),
\end{equation}
where $k_{p,z}(\xi)=K_z(\xi)/\|k_{p,z}\|_{A^p(\om)}$. 

Let $s=\frac{2}{p}-\frac{2}{q}$. Then, by \eqref{eq-V10} and successively \eqref{Eq-gamma}, \eqref{eqn:asymptau} and\eqref{eqn:RK-Diag1n}, we get
\begin{align*}
	\left(\|K_z\|^{2s}_{A^2(\om)}\psi_\om(z)|g'(z)|\right)^q&\lesssim\frac{\|K_z\|^{2qs}_{A^2(\om)}}{\tau(z)^{2}\omega(z)^{1-\frac{q}{p}}}\int_{D_{\delta}(z)}\frac{|g'(\xi)|^q}{(1+\varphi'(\xi))^{q}}\,\omega(\xi)^{1-\frac{q}{p}} \,dA(\xi)\\
	&\lesssim\frac{1}{\tau(z)^{2q/p}}\int_{D_{\delta}(z)}\frac{|g'(\xi)|^q}{(1+\varphi'(\xi))^{q}}\,\omega(\xi)^{1-\frac{q}{p}} \,dA(\xi)\\
	&\lesssim\int_{D_{\delta}(z)}|k_{p,z}(\xi)|^q\frac{|g'(\xi)|^q}{(1+\varphi'(\xi))^{q}}\,\omega(\xi)\, dA(\xi)\\
	&\lesssim GB^{id}_{0,p,q}(g')(z)< \infty.
\end{align*}
Thus, to prove that the function $g'$ vanishes on $\D$, it is enough to show that $\|K_z\|^{2s}_{A^2(\om)}\psi_\om(|z|)$ goes to infinity as $|z|\to 1$. Indeed, by \eqref{Eq-NE} and \eqref{jcn}, we have 
$$
\|K_z\|^{2s}_{A^2(\om)}\psi_\om(|z|)\asymp \frac{\tau(z)^{2(1-s)}}{(1-|z|)^{t}\omega(z)^{s}},
$$
and so,
$$
	\lim_{\substack{|z|\to 1}} \|K_z\|^{2s}_{A^2(\om)}\psi_\om(z)=\infty
$$
because of Lemma 2.3 in \cite{PP1}.

\subsection{Proof of Corollary \ref{cor2}} (III) Let $q < p$, and suppose that $GB^{id}_{0,p,q}(g')\in L^{p/(p-q)}(\D,d\lambda)$. Then
\begin{equation*}
\begin{split}
GB^{id}_{0,p,q}(g')(z)&\gtrsim\int_{D_{\delta}(z)}|k_{p,z}(\xi)|^q\frac{|g'(\xi)|^q}{(1+\varphi'(\xi))^{q}}\,\omega(\xi) \,dA(\xi)\\
&\gtrsim\tau(z)^{-2q/p}\int_{D_{\delta}(z)}\frac{|g'(\xi)|^q}{(1+\varphi'(\xi))^{q}}\,\omega(\xi)^{\frac{p-q}{p}} \,dA(\xi),
\end{split}
\end{equation*}
and so it follows from the assumption that
\begin{equation}\label{ef}
\begin{split}
&\int_{\D}\left(\tau(z)^{-2}\int_{D_{\delta}(z)}\frac{|g'(\xi)|^q}{(1+\varphi'(\xi))^{q}}\,\omega(\xi)^{\frac{p-q}{p}} \,dA(\xi)\right)^{\frac{p}{p-q}} dA(z)\\
&\lesssim \int_{\D}\left(GB^{id}_{0,p,q}(g')(z)\right)^{\frac{p}{p-q}}d\lambda(z)<+\infty.
\end{split}
\end{equation}
Thus, using \eqref{littleW}, we get
$$
\|g\|_{A^{pq/(p-q)}(\om)}\lesssim \int_{\D}\left(\tau(z)^{-2}\int_{D_{\delta}(z)}\frac{|g'(\xi)|^q}{(1+\varphi'(\xi))^{q}}\,\omega(\xi)^{\frac{p-q}{p}} \,dA(\xi)\right)^{\frac{p}{p-q}} dA(z).
$$
This completes the proof of Corollary \ref{cor2}.

\end{document}